\documentclass[a4paper,11pt]{article}
\usepackage{makeidx}
\usepackage[english]{babel}
\usepackage[latin1]{inputenc}
\usepackage{amsfonts}
\usepackage{bm}
\usepackage{amssymb}
\usepackage{latexsym}
\usepackage{amsmath}
\usepackage{amscd}
\usepackage{amsthm}
\usepackage{hyperref}
\usepackage{footnote}
\usepackage{rotating}
\usepackage{graphicx}
\usepackage{pifont}
\usepackage{curves}
\usepackage{fancyhdr}
\usepackage{epsfig}
\usepackage{pstricks}
\usepackage{pst-tree}
\usepackage{subfigure}
\usepackage{epic}
\usepackage{alltt}
\usepackage[all]{xy}
\usepackage{appendix}
\usepackage{ulem}
\usepackage{nth}

\usepackage{color}
\definecolor{darkgreen}{rgb}{0,0.4,0}
\definecolor{MyDarkBlue}{rgb}{0,0.08,0.50}
\definecolor{BrickRed}{rgb}{0.65,0.08,0}

\hypersetup{
colorlinks=true,       
    linkcolor=blue,          
    citecolor=red,        
    filecolor=BrickRed,      
    urlcolor=darkgreen        
}

\title{Extinction probabilities for a distylous plant population modeled by an inhomogeneous random walk on the positive quadrant}

\author{Pauline Lafitte-Godillon\footnote{Ecole Centrale de Paris, Grande Voie des Vignes,
92295 Ch\^atenay-Malabry C\'edex (France) \& EPI SIMPAF, INRIA Lille
Nord-Europe, 40 avenue Halley, 59650 Villeneuve d'Ascq (France)\newline Email: \url{pauline.lafitte@ecp.fr}}, Kilian Raschel\footnote{CNRS and Universit\'e de Tours, Laboratoire de Math\'ematiques et Physique Th\'eorique, Parc de Grandmont, 37200 Tours (France)\newline Email: \url{kilian.raschel@lmpt.univ-tours.fr}}, Viet Chi Tran\footnote{Universit\'e des Sciences et Technologies Lille 1, Laboratoire Paul Painlev\'e,
59655 Villeneuve d'Ascq C\'edex (France)\newline Email: \url{chi.tran@math.univ-lille1.fr}}}

\date{\today}

\numberwithin{equation}{section}

\renewcommand{\geq}{\geqslant}
\renewcommand{\leq}{\leqslant}

\def\N{\mathbb{N}}

\def\P{\mathbb{P}}
\def\R{\mathbb{R}}
\def\E{\mathbb{E}}

\def\ind{{\mathchoice {\rm 1\mskip-4mu l} {\rm 1\mskip-4mu l}
{\rm 1\mskip-4.5mu l} {\rm 1\mskip-5mu l}}}

\def\ie{\textit{i.e.}\ }
\def\eg{\textit{e.g.}\ }
\def\etal{\textit{et al.}\ }

\newcommand{\be} {\begin{equation}}
\newcommand{\ee} {\end{equation}}
\newcommand{\bea} {\begin{eqnarray}}
\newcommand{\eea} {\end{eqnarray}}
\newcommand{\Bea} {\begin{eqnarray*}}
\newcommand{\Eea} {\end{eqnarray*}}

\newtheorem*{Thm*}{Theorem}

\theoremstyle{definition} 
\theoremstyle{definition} 
\theoremstyle{definition} 

\theoremstyle{remark}\newtheorem{Rque}{Remark}

\theoremstyle{plain}
\newtheorem{theorem}{Theorem}[section]
\newtheorem{proposition}[theorem]{Proposition}
\newtheorem{lemma}[theorem]{Lemma}

\begin{document}

\maketitle

\begin{abstract}In this paper, we study a  flower population in
  which self-reproduction is not permitted. Individuals are diploid, {that is,
    each cell contains two sets of chromosomes}, and {distylous, that is, two alleles, A and a, can
  be found at the considered locus S}. Pollen and ovules of flowers with the same
  genotype at locus S cannot mate. This prevents the pollen of a given flower to
  fecundate its {own} stigmata. Only genotypes AA and Aa can be maintained in
  the population, so that the latter can be described by a random walk in the
  positive quadrant whose components are the number of individuals of each
  genotype. This random walk is not homogeneous and its transitions depend on
  the location of the process. We are interested in the computation of the
  extinction probabilities, {as} extinction happens when one of the axis is
  reached by the process. These extinction probabilities, which depend on the
  initial condition, satisfy a doubly-indexed recurrence equation that cannot be
  solved directly. {Our contribution is twofold : on the one hand, we obtain
    an explicit, though intricate, solution through the study of the PDE solved
    by the associated generating function. On the other hand, we provide
    numerical results comparing stochastic and deterministic approximations of
    the extinction probabilities.}
\end{abstract}


\noindent \textbf{AMS:} 60G50; 60J80; 35Q92; 92D25

\section{Introduction}

We consider the model of flower population without pollen limitation introduced
in Billiard and Tran \cite{billiardtran}. The flower reproduction is sexual:
plants produce pollen that may fecundate the stigmata of other plants. We are
interested in self-incompatible reproduction, where an individual can reproduce
only with compatible partners. In particular, self-incompatible reproduction
prevents the fecundation of a plant's stigmata by its own pollen. Each plant is
diploid and characterized by the two alleles that it carries at the locus $S$,
which decide on the possible types of partners with whom the plant may reproduce
(as it encodes the recognition proteins present on the pollen and stigmata of
the plant). We consider the distyle case with only two possible types for the
alleles, $A$ or $a$. The plants thus have genotypes $AA$, $Aa$ or $aa$. The only
interesting case is when $A$ is dominant over $a$ (see \cite{billiardtran}), and
we restrict to this case in this work.  Then, the phenotype, \ie the type of
proteins carried by the pollen and stigmata, of individuals with genotypes $AA$
(resp.\ $Aa$ and $aa$) is $A$ (resp.\ $A$ and $a$). Only pollen and stigmata
with different proteins can give viable seeds, \ie pollen of a plant of
phenotype $A$ can only fecundate stigmata of a plant of phenotype $a$ and
vice-versa. It can be seen that seeds $AA$ cannot be created, since the genotype
of individuals of phenotype $a$ is necessarily $aa$ that combine only with
individuals of phenotype $A$ that have genotypes $AA$ or $Aa$, therefore we can
consider without restriction populations consisting only of individuals of
genotypes $Aa$ and $aa$. Each viable seed is then necessarily of genotype $Aa$
or $aa$ with probability $1/2$. It is assumed that ovules are produced in
continuous time at rate $r>0$ and that each ovule is fecundated to give a seed,
provided there exists compatible pollen in the population. The lifetime of each
individual follows an exponential distribution with mean $1/d$, where $d>0$. In
all the article, we consider
\begin{equation}
\label{main_assumption}
     r>d
\end{equation}
which, we will see, is the interesting case.

\unitlength=1cm
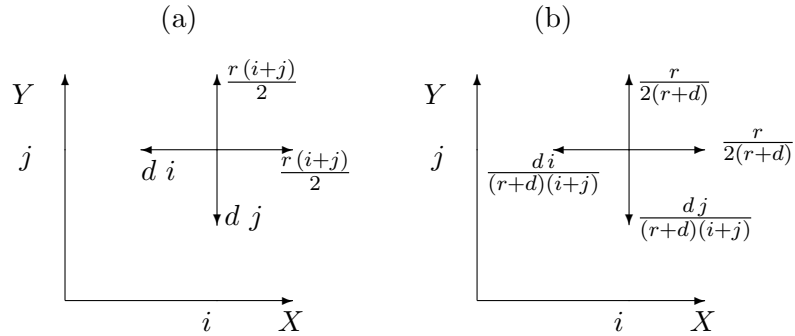
\begin{figure}[ht]
  \begin{center}
 \hspace{-0.5cm} \begin{tabular}{cc}
  (a) & (b)\\
  \\
    \begin{picture}(4,4)
    \put(1,1){\vector(1,0){3}}
    \put(1,1){\vector(0,1){3}}
    \put(3.8,0.6){$X$}
    \put(0.3,3.6){$Y$}
    \put(3,3){\vector(1,0){1}}
    \put(3,3){\vector(-1,0){1}}
    \put(3,3){\vector(0,1){1}}
    \put(3,3){\vector(0,-1){1}}
    \put(3.8,2.6){$\frac{r\, (i+j)}{2}$}
    \put(2,2.6){$d\ i$}
    \put(3.1,2){$d\ j$}
    \put(3.1,3.8){$\frac{r\, (i+j)}{2}$}
    \dottedline(1,3)(3,3)
    \dottedline(3,1)(3,3)
    \put(0.4,2.8){$j$}
    \put(2.8,0.6){$i$}
    \end{picture}\hspace{1cm}
&   \begin{picture}(4,4)
    \put(1,1){\vector(1,0){3}}
    \put(1,1){\vector(0,1){3}}
    \put(3.8,0.6){$X$}
    \put(0.3,3.6){$Y$}
    \put(3,3){\vector(1,0){1}}
    \put(3,3){\vector(-1,0){1}}
    \put(3,3){\vector(0,1){1}}
    \put(3,3){\vector(0,-1){1}}
    \put(4.2,3){$\frac{r}{2(r+d)}$}
    \put(1.1,2.6){$\frac{d\,i}{(r+d)(i+j)}$}
    \put(3.1,2){$\frac{d\,j}{(r+d)(i+j)}$}
    \put(3.1,3.8){$\frac{r}{2(r+d)}$}
    \dottedline{3}(1,3)(3,3)
    \dottedline{3}(3,1)(3,3)
    \put(0.4,2.8){$j$}
    \put(2.8,0.6){$i$}
    \end{picture}
    \end{tabular}
  \vspace{-0.3cm}
  \caption{\textit{(a) Transition rates for the continuous-time pure-jump Markov process $(X_t,Y_t)_{t\in \R_+}$. (b)~Transition probabilities of the embedded random walk, that we denote by $(X_t,Y_t)_{t\in \N}$ (here and throughout, $\N$ is the set $\{0,1,2,\ldots \}$), with an abuse of notation.}}\label{fig1}
  \end{center}
\end{figure}

\par Let us denote by $X_t$ and $Y_t$ the number of individuals of genotype $Aa$ (phenotype $A$) and $aa$ (phenotype $a$) at time $t\in \R_+$. The process $(X_t,Y_t)_{t\in \R_+}$ is a pure-jump Markov process with transitions represented in Fig.\ \ref{fig1}(a). A stochastic differential equation (SDE) representation of $(X_t,Y_t)_{t\in \R_+}$ is given in \cite{billiardtran}. Here we forget the continuous-time process, and we are interested in the embedded discrete-time Markov chain, which we denote, with an abuse of notation, by $(X_t,Y_t)_{t\in \N}$, and with transitions represented in Fig.\ \ref{fig1}(b):
\begin{eqnarray*}
 & \P_{i,j}[(X_1,Y_1)=(i-1,j)]=\dfrac{d\,i}{(r+d)(i+j)},\qquad    & \P_{i,j}[(X_1,Y_1)=(i+1,j)]=\frac{r}{2(r+d)},\\
  & \P_{i,j}[(X_1,Y_1)=(i,j-1)]=\dfrac{d\,j}{(r+d)(i+j)},\qquad         & \P_{i,j}[(X_1,Y_1)=(i,j+1)]=\frac{r}{2(r+d)},
\end{eqnarray*}
where $\P_{i,j}$ means that the process starts with the initial condition
$(X_0,Y_0)=(i,j)$. The main and profound difficulty is that this random walk is
not homogeneous in space, while techniques developed in the literature for
random walks on positive quadrants mostly focus on the homogeneous case (see \eg
Fayolle \etal \cite{fayollemalyshevmenshikov}, Klein Haneveld and Pittenger
\cite{kleinhaneveldpittenger}, Kurkova and Raschel \cite{kurkovaraschel},
Walraevens, van Leeuwaarden and Boxma \cite{walraevensleeuwaardenboxma}). We
introduce a generating function \eqref{defP} that satisfies here a partial
differential equation (PDE) of a new type that we solve. Although the
particularity of the problem is exploited, these techniques and the links
between probability and PDEs may be extended to carry out {general studies} of
inhomogeneous random walks in cones. The introduction of PDEs through generating
functions had been already used by Feller \cite{feller1} for a trunking problem
with an inhomogeneous random walk in dimension 1. To our knowledge, the case of
inhomogeneous random walks in the cone with absorbing boundaries has been left
open. In \cite{regesengupta}, the discriminatory processor-sharing queue is
considered but boundaries are not absorbing and the overall arrival rate is
constant, which is not the case in our model.

\par When one of the phenotype $A$ or $a$ disappears, reproduction becomes
impossible and {the extinction of  the system occurs}. We are interested in the probability of
extinction of $(X_t,Y_t)_{t\in \N}$ (or, equivalently, in that of
$(X_t,Y_t)_{t\in \R_+}$). Let us introduce the first time at which one of the
two types gets extinct:
\begin{equation}
\label{def_tau_0}
\tau_0=\inf\{t\in \N:\, X_t=0\mbox{ or }Y_t=0\}.
\end{equation}
For $i,j\in \N$, let us denote by
\begin{equation}
p_{i,j}=\P_{i,j}[\tau_0<\infty]\label{eq:pi}
\end{equation}
the absorption probabilities, and by
\begin{equation}
\label{defP}
P(x,y)=\sum_{i,j\geq 1}p_{i,j}x^i y^j
\end{equation}
their generating function. By symmetry arguments, we have, for all $i,j\in \N$,
\begin{equation}
p_{i,j}=p_{j,i}.\label{symmetry}
\end{equation}Moreover, for any $i,j\in\N$ such that $i=0$ or $j=0$, we have
\begin{equation}
p_{i,j}=1.\label{boundary_condition_dirichlet}
\end{equation}

In Section \ref{section:existence}, we will see that the $p_{i,j}$'s satisfy the Dirichlet problem associated with the following doubly-indexed recurrence equation
\begin{equation}
q_{i,j}=  \frac{d i}{(r+d)(i+j)}q_{i-1,j}+\frac{d j}{(r+d)(i+j)}q_{i,j-1}+\frac{r}{2(r+d)}q_{i,j+1}+\frac{r}{2(r+d)}q_{i+1,j}\label{eq3}
\end{equation}
and with the boundary condition \eqref{boundary_condition_dirichlet}. This
problem does not admit simple solutions. There is no uniqueness of solutions to
this problem.  Note that the constant sequence equal to $1$ is a
solution. However, we are interested in solutions that tend to $0$ as $i$ or $j$
tends to infinity, since, \cite{billiardtran} (see Proposition
\ref{prop:billiardtran} in this paper), estimates for $p_{i,j}$ were obtained
through probabilistic coupling techniques; they show that in the case
\eqref{main_assumption} we consider, $p_{i,j}$ is strictly less than $1$. In
fact, the $p_{i,j}$'s correspond to the smallest positive solution of the
Dirichlet problem, and are completely determined if we give the probabilities
$(p_{i,1})_{i\geq 1}$. We conclude the section with more precise estimates of
the absorption probabilities $p_{i,j}$ as the initial state $(i,j)$ goes to
infinity along one axis (Proposition \ref{AsymptoticBehaviorAP}). These new
estimates rely on Proposition \ref{prop:billiardtran} and on comparisons with
one-dimensional random walks. In Section \ref{section:green}, we consider the
generating function $P(x,y)$ associated with the $p_{i,j}$'s and show that it
satisfies a PDE, that has one and only one solution, that is computed
(Proposition \ref{explicit_one}) explicitly with a dependence on the
$(p_{i,1})_{i\geq 1}$, prompting us to use the name ``Green's function''. This
provides a new formulation of the solution of \eqref{eq3}, that is however
uneasy to work with numerically. Hence, in Section \ref{section:numerique}, we
propose two different approaches leading to numerical approximations of the
solution of the Dirichlet problem
\eqref{boundary_condition_dirichlet}--\eqref{eq3}, that are based on stochastic
and deterministic approaches.\\
In conclusion, we provide here several approaches to handle the extinction
probabilities of the inhomogeneous random walk $(X_t,Y_t)_{t\in \R_+}$ of our
problem. Estimates from \cite{billiardtran} are recalled and the recurrence
equation \eqref{eq3} is solved numerically and theoretically, pending
further investigation {of the PDE formulation}.

\section{Existence of a solution}\label{section:existence}

\subsection{Dirichlet problem}
We first establish that the extinction probabilities $p_{i,j}$'s \eqref{eq:pi} solve the Dirichlet problem \eqref{boundary_condition_dirichlet}--\eqref{eq3}.\\

\begin{proposition}\label{prop_dirichlet}
\begin{enumerate}
\item \label{fp} The extinction probabilities $(p_{i,j})_{i,j\geq 1}$ are solutions to the Dirichlet problem (\ref{eq3}) with boundary condition (\ref{boundary_condition_dirichlet}). Uniqueness of the solution may not hold, but the extinction probabilities $(p_{i,j})_{i,j\in \N}$ define
the smallest positive solution to this problem.
\item \label{sp} Let the probabilities $(p_{i,1})_{i\geq 1}$ be given. Then the probabilities $(p_{i,j})_{i,j\geq 1}$ are completely determined.
\end{enumerate}
\end{proposition}

\bigskip
\begin{proof}
We begin with Point \ref{fp}. Equation (\ref{eq3}) is obtained by using the strong Markov property at the time of the first event.
Let us denote by $K$ the transition kernel of the discrete-time Markov chain $(X_t,Y_t)_{t\in \N^*}$; we have:
\begin{multline*}
Kf(i,j)\\
=(f(i+1,j)+f(i,j+1))\frac{r}{2(r+d)}+f(i-1,j)\frac{di}{(r+d)(i+j)}+f(i,j-1)\frac{dj}{(r+d)(i+j)}.
\end{multline*}
Following classical proofs (\eg \cite{baldimazliakpriouret,revuz}), the extinction probabilities $(p_{i,j})_{i,j\in \N}$ satisfy the equation:
\begin{equation}
\forall i,j\in \N^*,\quad f(i,j)=Kf(i,j)\qquad \mbox{ and }\qquad \forall i,j\in \N,\quad f(i,0)=f(0,j)=1.\label{pb_dirichlet}
\end{equation}
The constant solution equal to $1$ is a solution to (\ref{pb_dirichlet}). Let us prove that $(p_{i,j})_{i,j\in \N}$ is the smallest positive solution to \eqref{pb_dirichlet}. Let $f$ be another positive solution. Let us consider $M_t=f(X_{\inf\{t, \tau_0\}},Y_{\inf\{t, \tau_0\}})$, with $\tau_0$ defined in \eqref{def_tau_0}. Denoting by $(\mathcal{G}_t)_{t\in \N}$ the filtration of $(M_t)_{t\in \N}$, we have:
\begin{align*}
\E\big[M_{t+1}\, |\, \mathcal{G}_t\big]= \ & \E\big[M_{t+1}\ind_{\tau_0 \leq t}+M_{t+1}\ind_{\tau_0>t}\, |\, \mathcal{G}_t\big]\\=\ &  \E\big[M_{t}\ind_{\tau_0\leq t}+f(X_{t+1},Y_{t+1})\ind_{\tau_0>t}\, |\, \mathcal{G}_t\big]\nonumber\\
= \ & M_{t}\ind_{\tau_0\leq t}+\ind_{\tau_0>t}\E\big[f(X_{t+1},Y_{t+1})\, |\, \mathcal{G}_t\big]\\
= \ &  M_{t}\ind_{\tau_0\leq t}+\ind_{\tau_0>t}Kf(X_t,Y_t)\nonumber\\
= \ & M_{t}\ind_{\tau_0\leq t}+\ind_{\tau_0>t}f(X_t,Y_t)=M_t.
\end{align*}
Hence $(M_t)_{t\in \N}$ is a martingale, which converges on $\{\tau_0<\infty\}$ to $f(X_{\tau_0},Y_{\tau_0})=1$ (see the boundary condition in \eqref{pb_dirichlet}).
Thus by using the positivity of $f$ and Fatou's lemma, we obtain that for every $i,j\in \N$:
\begin{equation*}
f(i,j)=\E_{i,j}\big[M_0\big]=\lim_{t\rightarrow \infty}\E_{i,j}\big[M_t\big]\geq \E\big[\liminf_{t\rightarrow \infty}M_t \ind_{\tau_0<\infty}\big]=\E_{i,j}\big[\ind_{\tau_0<\infty}\big]=p_{i,j}.
\end{equation*}This concludes the proof of Point \ref{fp}.

\medskip

Let us now consider Point \ref{sp}. Assume that the probabilities $(p_{i,1})_{i\geq 1}$ are given, and let us prove, by recursion, that every $p_{i,j}$ can be computed. By symmetry, we only need to prove that this is the case for $i\geq j$. Assume\\

\begin{center}(Hrec $j$): for $j\in \N^*$ all the $p_{k,\ell}$'s for $\ell\leq j$ and $k\geq \ell$ can be computed from the $p_{i,1}$'s
\end{center}
\bigskip
\noindent and let us prove that we can determine the $p_{i,j+1}$'s for $i\geq j+1$. From (\ref{eq3}) we get:
\begin{align}
p_{i,j+1}=\frac{2(r+d)}{r}p_{i,j}-\frac{2\,d\,i}{r(i+j)}p_{i-1,j}-\frac{2\,d\,j}{r(i+j)}p_{i,j-1}-p_{i+1,j}.\label{etape7}
\end{align}All the terms in the r.h.s.\ of (\ref{etape7}) are known by (Hrec $j$), and hence $p_{i,j+1}$ can be computed for any $i\geq j+1$. This concludes the recursion.
\end{proof}

The following result shows that there is almost sure extinction in the case $r\leq d$. In the interesting case $r>d$, it also shows that there is a nontrivial solution to the Dirichlet problem \eqref{boundary_condition_dirichlet}--\eqref{eq3}.\\

\begin{proposition}[Proposition 9 of \cite{billiardtran}]\label{prop:billiardtran}We have the following regimes given the parameters $r$ and $d$:
\begin{enumerate}
\item \label{ppff} If $r\leq d$, we have almost sure extinction of the population.
\item \label{ppss} If $r>d(>0)$, then there is a strictly positive survival probability. Denoting by $(i,j)$ the initial condition, we have:
\begin{equation*}
\left(\frac{d}{r}\right)^{i+j}\leq p_{i,j}\leq \left(\frac{d}{r}\right)^{i}+\left(\frac{d}{r}\right)^{j}-\left(\frac{d}{r}\right)^{i+j}.
\end{equation*}
\end{enumerate}
\end{proposition}
In Point \ref{ppss}, only bounds, and no explicit formula, are available for the extinction probability $p_{i,j}$. The purpose of this article is to address \eqref{eq3} by considering the Green's function $P(x,y)$ introduced in \eqref{defP}.

\subsection[Absorption probability as the initial state goes to infinity along one axis]{Asymptotic behavior of the absorption probability as the initial state goes to infinity along one axis}

In this part, using the result of Proposition \ref{prop:billiardtran}, we provide more precise estimates of the asymptotic behavior of the absorption probability $p_{1,j} = p_{j,1}$ when $j\to \infty$. In particular, these estimates will  be very useful when we tackle the deterministic numerical simulations (see Section \ref{section:calculnumerique}). \\

\begin{proposition}
\label{AsymptoticBehaviorAP}
If $j\to\infty$, then
\begin{equation}
\label{conc1}
     p_{1,j} = p_{j,1}= \frac{2d}{r}\frac{1}{j}-\frac{2d(r^2+dr+2d^2)}{r^2(r+d)}\frac{1}{j^2}+O\left(\frac{1}{j^3}\right).
\end{equation}
\end{proposition}

\begin{proof}
In addition to $\tau_0$, defined in \eqref{def_tau_0}, we introduce
\begin{equation*}
     S = \inf\{t\in\N : Y_t=0\},
     \qquad
     T = \inf\{t\in\N : X_t=0\},
\end{equation*}
the hitting times of the horizontal axis and vertical axis, respectively. Note that we have $\tau_0 = \inf\{S,T\}$. Let $f:\N \to \N$ be a function such that $f(j)<j$ for any $j\geq 1$. In the sequel, we will choose $f(j) = \lfloor \epsilon j\rfloor $, with $\epsilon\in (0,1)$ and where $\lfloor . \rfloor$ denotes the integer part). We obviously have the identity:
\begin{equation}
\label{sep_id}
     p_{1,j} = \mathbb P_{(1,j)}[\tau_0<\infty] = \mathbb P_{(1,j)}[\tau_0\leq f(j)] + \mathbb P_{(1,j)}[f(j)<\tau_0<\infty].
\end{equation}
To prove Proposition \ref{AsymptoticBehaviorAP}, we shall give estimates for both terms in the r.h.s. of \eqref{sep_id}.

\medskip

\noindent{\it First step:} Study of $\mathbb P_{(1,j)}[\tau_0\leq f(j)]$.
Since $f(j)<j$, it is impossible, starting from $(1,j)$, to reach the horizontal axis before time $f(j)$, and we have $\mathbb P_{(1,j)}[\tau_0\leq f(j)]=\mathbb P_{(1,j)}[T\leq f(j)]$. In order to compute the latter probability, we introduce two one-dimensional random walks on $\N$, namely $X^-$ and $X^+$, which are killed at $0$, and which have the jumps
\begin{equation*}
     \mathbb P_i[X^\pm_1=i-1] = q_i^\pm,\quad
     \mathbb P_i[X^\pm_1=i+1] = p_i^\pm,\quad
     \mathbb P_i[X^\pm_1=i] = r_i^\pm,\quad
     q_i^\pm+p_i^\pm+r_i^\pm = 1,
\end{equation*}
where
\begin{equation}
\label{expression-transition-probabilities}
     q_i^\pm = \frac{d i}{(r+d)(i+j\mp f(j))},\quad
     p_i^\pm = \frac{r}{2(r+d)}.
\end{equation}Both $X^-$ and $X^+$ are (inhomogeneous) birth-and-death processes on $\mathbb N$. These random walks are implicitely parameterized by $j$. If $T^\pm=\inf\{ t\in \N : X^\pm_t =0\}$, then
\begin{equation}
\label{encadrement}
     \mathbb P_{1}[T^-\leq f(j)]\leq\mathbb P_{(1,j)}[T\leq f(j)]\leq\mathbb P_{1}[T^+\leq f(j)].
\end{equation}The quantities $\mathbb P_{1}[T^\pm\leq f(j)]$ are computable: we shall prove that
\begin{equation}
\label{bbttcc}
     \mathbb P_1[T^\pm\leq f(j)] = \frac{2d}{r}\frac{1}{(j\mp f(j))}-\frac{2d(r^2+dr+2d^2)}{r^2(r+d)}\frac{1}{(j\mp f(j))^2}+O\left(\frac{1}{(j\mp f(j))^3}\right).
\end{equation}
The main idea for proving \eqref{bbttcc} is that the $q_i^\pm$ being very small as $j\to\infty$, the only paths which will significantly contribute to the probability $\mathbb P_{1}[T^\pm\leq f(j)]$ are the ones with very few jumps to the right. Let us define
\begin{align*}
     \Lambda_{t}^{\pm}( p ) &= \{\text{the chain $X^\pm$ makes exactly $p$ jumps to the right between $0$ and $t$}\}\\
    &= \{\text{there exist $0\leq q_1<\cdots <q_p\leq t-1$ such that }\\
    & \hspace{4.8cm}\text{ $X_{q_1+1}^\pm-X_{q_1}^\pm = \cdots = X_{q_p+1}^\pm-X_{q_p}^\pm = 1$}\}.
\end{align*}
We are entitled to write
\begin{multline}
\label{decomposition}
     \mathbb P_1[T^\pm\leq f(j)] = \mathbb P_1[T^\pm\leq f(j),\, \Lambda_{f(j)}^{\pm}(0)] + \mathbb P_1[T^\pm\leq f(j),\, \Lambda_{f(j)}^{\pm}(1)]\\ +  \mathbb P_1[T^\pm\leq f(j),\, \cup_{p\geq 2}\Lambda_{f(j)}^{\pm}( p)],
\end{multline}
and we now separately analyze the three terms in the right-hand side of \eqref{decomposition}. First:
\begin{align}
     \mathbb P_1[T^\pm\leq f(j),\, \Lambda_{f(j)}^{\pm}(0)] = & \sum_{k=1}^{f(j)}\mathbb P_1[T^\pm = k,\, \Lambda_{f(j)}^{\pm}(0)]
     =    \sum_{k=1}^{f(j)} {(r_1^\pm)}^{k-1}q_1^\pm \nonumber\\
     = & \frac{q_1^\pm}{1-r_1^\pm} \big(1-(r_1^\pm)^{f(j)}\big)
= \frac{q_1^\pm}{1-r_1^\pm}(1+O({(r_1^\pm)}^{f(j)})).\label{first-term-decompo}
\end{align}A Taylor expansion of $q_1^\pm/(1-r_1^\pm)$ according to the powers of $1/(j\mp f(j))$ together with the fact that $(r_1^\pm)^{f(j)}=o(1/(j\mp f(j))^3)$ provides that:
\begin{equation}
     \mathbb P_1[T^\pm\leq f(j),\, \Lambda_{f(j)}^{\pm}(0)] = \frac{2d}{r(j\mp f(j))}\left(1-\frac{1+\frac{2d}{r}}{j\mp f(j)}+O\left(\frac{1}{(j\mp f(j))^2}\right)\right).\label{etape8}
\end{equation}
We now consider the second term in the right-hand side of \eqref{decomposition}. On the event $\Lambda_{f(j)}^{\pm}(1)$, $X^\pm$ first stays a time $k_1$ at $1$, then jumps to $2$, where it remains $k_2$ unit of times; it next goes to $1$, and, after a time $k_3$, jumps to $0$. Further, since $T^\pm\leq f(j)$, we have $k_1+1+k_2+1+k_3+1\leq f(j)$. Denoting by $\widetilde{k}_1=k_1+k_3$ the time spent in position $1$, we thus have:
\begin{align}
\lefteqn{     \mathbb P_1[T^\pm\leq f(j),\, \Lambda_{f(j)}^{\pm}(1)] }\nonumber\\
=& \sum_{k_1+k_2+k_3\leq f(j)-3} (r_1^\pm)^{k_1}p_1^\pm (r_2^\pm)^{k_2} q_2^\pm (r_1^\pm)^{k_3}q_1^\pm = p_1^\pm q_1^\pm q_2^\pm \sum_{\widetilde k_1+k_2\leq f(j)-3} (r_1^\pm)^{\widetilde k_1} (r_2^\pm)^{k_2}\nonumber\\
     = &  \frac{p_1^\pm q_1^\pm q_2^\pm}{(1- r_1^\pm)(1-r_2^\pm)}(1+O({(r_1^\pm \vee r_2^\pm )}^{f(j)-2})).\label{second-term-decompote}
\end{align}As for the first term, using the fact that ${(r_1^\pm \vee r_2^\pm )}^{f(j)-2}=o(1/(j\mp f(j))^3)$ and a Taylor expansion according to the powers of $1/(j\mp f(j))$ gives that:
\begin{align}
\mathbb P_1[T^\pm\leq f(j),\, \Lambda_{f(j)}^{\pm}(1)] =\frac{4d^2}{r(r+d)(j\mp f(j))^2}\left(1+O\left(\frac{1}{(j\mp f(j))}\right)\right).
\end{align}
Finally, let us consider the third term $\mathbb P_1[T^\pm\leq f(j),\, \cup_{p\geq 2}\Lambda_{f(j)}^{\pm}( p ) ]$. On $\cup_{p\geq 2}\Lambda_{f(j)}^{\pm}( p )$, the two first jumps to the right are either from $1$ to $2$ and $2$ to $3$, or twice from $1$ to $2$. Thus, extinction means that there is at least 3 jumps from $3$ to $2$, $2$ to $1$ and $1$ to $0$ or two jumps from $2$ to $1$ and $1$ to $0$. Since $q_i^{\pm}$ is an increasing function of $i$, we deduce:
\begin{multline}
     \mathbb P_1[T^\pm\leq f(j),\, \cup_{p\geq 2}\Lambda_{f(j)}^{\pm}( p )  ] \\ \leq  \left(q_3^{\pm}\right)^3
     =  \left(\frac{d}{r+d}\right)^3\left(\frac{3}{3+j\mp f(j)}\right)^3 = O\left(\frac{1}{(j\mp f(j))^3}\right).
\label{third-term-decompo}
\end{multline}
From \eqref{decomposition}, \eqref{first-term-decompo}, \eqref{second-term-decompote} and \eqref{third-term-decompo}, we obtain \eqref{bbttcc}.

\medskip

\noindent{\it Second step:} Study of $ \mathbb P_{(1,j)}[f(j)<\tau_0<\infty]$.
\begin{align}
\lefteqn{\mathbb P_{(1,j)}[f(j)<\tau_0<\infty]}\nonumber\\
= &    \sum_{k,\ell\geq 1} \mathbb P_{(1,j)}[f(j)<\tau_0<\infty\vert (X_{f(j)},Y_{f(j)})=(k,\ell)]\mathbb P_{(1,j)}[(X_{f(j)},Y_{f(j)})=(k,\ell)]\nonumber\\
 = &
  \sum_{k,\ell\geq 1}p_{k,\ell}\mathbb P_{(1,j)}[(X_{f(j)},Y_{f(j)})=(k,\ell)],\label{before_split}
\end{align}by using the strong Markov property.
Introduce now a function $g:\N \to \N$ such that $f(j)+g(j)<j$ for any $j\geq 1$. We can split \eqref{before_split} into
\begin{equation}
\label{probab_split}
     \sum_{k,\ell\geq g(j)}p_{k,\ell}\mathbb P_{(1,j)}[(X_{f(j)},Y_{f(j)})=(k,\ell)]+\sum_{\substack{\substack{g(j)\geq k \geq 1\\\text{and/or}}\\g(j) \geq \ell\geq 1}}p_{k,\ell}\mathbb P_{(1,j)}[(X_{f(j)},Y_{f(j)})=(k,\ell)].
\end{equation}
With Proposition \ref{prop:billiardtran} we obtain the following upper bound for the first sum in \eqref{probab_split}:
\begin{equation*}
     \sum_{k,\ell\geq g(j)}p_{k,\ell}\mathbb P_{(1,j)}[(X_{f(j)},Y_{f(j)})=(k,\ell)]\leq 2\left(\frac{d}{r}\right)^{g(j)}.
\end{equation*}
In particular, if we choose $g$ such that as $j\to\infty$, $g(j)\to\infty$ fast enough, then clearly the term above is negligible compared to \eqref{bbttcc}.
For the second sum in \eqref{probab_split},
\begin{equation*}
     \sum_{\substack{\substack{g(j)\geq k \geq 1\\\text{and/or}}\\g(j) \geq \ell\geq 1}}p_{k,\ell}\mathbb P_{(1,j)}[(X_{f(j)},Y_{f(j)})=(k,\ell)]=\sum_{g(j)\geq k\geq 1}p_{k,\ell}\mathbb P_{(1,j)}[(X_{f(j)},Y_{f(j)})=(k,\ell)],
\end{equation*}
since by assumption $j-f(j)>g(j)$ so that $Y$ cannot reach values $\ell \leq g(j)$ in $f(j)$ steps. Then we have
\begin{align}
     \sum_{g(j)\geq k\geq 1}p_{k,\ell}\mathbb P_{(1,j)}[(X_{f(j)},Y_{f(j)})=(k,\ell)]&\leq \sum_{g(j)\geq k\geq 1}\mathbb P_{(1,j)}[(X_{f(j)},Y_{f(j)})=(k,\ell)]\nonumber\\
     &\leq \mathbb P_{(1,j)}[0\leq X_{f(j)}\leq g(j)].\label{lp}
\end{align}
To obtain an upper bound for \eqref{lp} we are going to use, again, a one-dimensional random walk. Introduce $\widetilde X$, a random walk on $\N$ which is killed at $0$, homogeneous on $\N^*$ with jumps
\begin{equation*}
     \mathbb P_k[\widetilde X_1=k-1] = \widetilde q,\quad
     \mathbb P_k[\widetilde X_1=k+1] = \widetilde p,\quad
     \mathbb P_k[\widetilde X_1=k] = \widetilde r,\quad
     \widetilde q+\widetilde p+\widetilde r = 1,
\end{equation*}
where
\begin{equation*}
     \widetilde q = \frac{d (1+f(j))}{(r+d)(1+j-2 f(j))},\quad
     \widetilde p = \frac{r}{2(r+d)}.
\end{equation*}This walk is again parameterized by $j$.
By construction of $\widetilde X$, we have
\begin{equation*}
     \mathbb P_{(1,j)}[0\leq X_{f(j)}\leq g(j)]\leq \mathbb P_{1}[\widetilde X_{f(j)}\leq g(j)].
\end{equation*}
Denoting by $\widetilde m$ and $\widetilde \sigma^2$ the mean and the variance of $(\widetilde{X}_2-\widetilde{X}_1)$, respectively (they could easily be computed), we can write
\begin{equation*}
     \mathbb P_{1}[\widetilde X_{f(j)}\leq g(j)]=\mathbb P_0\left[\frac{\widetilde X_{f(j)}-\widetilde mf(j)}{\widetilde \sigma \sqrt{f(j)}}\leq \frac{g(j)-1-\widetilde m f(j)}{\widetilde \sigma \sqrt{f(j)}} \right].
\end{equation*}
By a suitable choice of the functions $f$ and $g$, for instance $f(j)=\lfloor  \epsilon j\rfloor$ with $\epsilon \in (0,1)$ and $g(j)=\lfloor  j^{3/4}\rfloor$, the central limit theorem gives that the latter is negligible compared to $\P_{(1,j)} [\tau_0\leq f(j)]$. For this last term, using \eqref{bbttcc}, \eqref{encadrement} and letting $\epsilon$ tend to 0 provides \eqref{conc1}. The proof is concluded.
\end{proof}

Let us make some remarks on possible extensions of Proposition \ref{AsymptoticBehaviorAP}. \\

\begin{Rque}
1. The proof of Proposition \ref{AsymptoticBehaviorAP} can easily be extended to the asymptotic of $p_{i,j}$ as $j\to\infty$, for any fixed value of $i$. In particular, we have the following asymptotic behavior:
\begin{equation}
      p_{i,j} = \left(\frac{2d}{r}\right)^i \frac{i !}{j^i}+O\left(\frac{1}{j^{i+1}}\right).\label{conc}
\end{equation}
2. It is possible to generalize \eqref{decomposition} by
\begin{equation}
\P_1[T^\pm \leq f(j)]=\sum_{p=0}^{k-1} \P_1[T^\pm \leq f(j),\ \Lambda_p^{f(j)}]+\P_1[T^\pm \leq f(j),\ \cup_{p\geq k}\Lambda_p^{f(j)}].
\end{equation}We can show as in the proof of Proposition \ref{AsymptoticBehaviorAP} that $\P_1[T^\pm \leq f(j),\ \cup_{p\geq k}\Lambda_p^{f(j)}]\leq (q_{k+1}^\pm)^{k+1}=o(1/j^{k+1})$ (see \eqref{third-term-decompo}) and that the probabilities $ \P_1[T^\pm \leq f(j),\ \Lambda_p^{f(j)}]$ admit Taylor expansions in powers of $1/(j\mp f(j))$ where the development for the $p^{\rm{th}}$ probability has a main term in $1/(j\mp f(j))^{p+1}$. This allows us to push the developments in \eqref{bbttcc} to higher orders.

For instance, the next term in \eqref{conc1} can be obtained by a long computation. First, we generalize \eqref{second-term-decompote} by writing that
\begin{multline}
     \mathbb P_1[T^\pm\leq f(j),\, \Lambda_2^{f(j)}] \\
     = q_1^\pm \big((p_1^\pm)^2 (q_2^\pm)^2+p_1^\pm p_2^\pm q_3^\pm q_2^\pm \big) \sum_{\widetilde k_1+ \widetilde{k}_2 + \widetilde{k}_3\leq f(j)-5} (r_1^\pm)^{\widetilde k_1} (r_2^\pm)^{\widetilde{k}_2}(r_3^\pm)^{\widetilde{k}_3}
\end{multline}where $\widetilde{k}_1$, $\widetilde{k}_2$ and $\widetilde{k}_3$ are the times spent by the random walk in the states $1$, $2$ and $3$. Then, pushing further the Taylor expansion leads to:
\begin{multline*}
     p_{1,j} = \frac{2d}{r}\frac{1}{j\mp f(j)}-\frac{2d(r^2+dr+2d^2)}{r^2(r+d)}\frac{1}{(j\mp f(j))^2}\\
     +d\ \Big(\frac{2}{r}\big(1+\frac{2d}{r}\big)^2-\frac{24d\big(\frac{r}{2}+d\big)}{r^2(r+d)}+\frac{5r^2d^2}{2(r+d)^2\big(\frac{r}{2}+d\big)^3}\Big)\frac{1}{(j\mp f(j))^3}+O\left(\frac{1}{(j\mp f(j))^4}\right).
\end{multline*}
\end{Rque}

\section{Green's function}\label{section:green}

\subsection{A functional equation for the Green's function}

In this section, we consider the Green's function $P(x,y)$ defined in \eqref{defP}
associated with a solution of \eqref{eq3} in the same spirit as what can be
found in Feller \cite[Ch.\ XVII]{feller1}. We show that it
satisfies a non-classical linear PDE that can be solved (see Proposition \ref{explicit_one}).\\

\begin{proposition}\label{prop:eqP}(i) The function $P(x,y)$ satisfies formally:
  \begin{align}
    AP(x,y)=h(x,y,P),\label{equation_P}
  \end{align}where:
  \begin{align}
    AP(x,y)= & Q(x,y)\frac{\partial P}{\partial x}(x,y)+Q(y,x)\frac{\partial P}{\partial y}(x,y)+R(x,y)P(x,y),\label{def:A}\\
    Q(x,y)= & (r+d)x-\frac{r}{2}-\frac{r}{2}\frac{x}{y}-d\,x^2,\nonumber\\
    R(x,y)= & \frac{r}{2x}+\frac{r}{2y}-dx-dy,\nonumber
  \end{align}and where:
 \begin{align}
    h(x,y,P)= & -\frac{r}{2}\Bigg(x\frac{\partial^2 P}{\partial x\partial y}(x,0)+y\frac{\partial^2 P}{\partial y\partial x}(0,y)\Bigg) +d\, xy\Bigg(\frac{1}{1-x}+\frac{1}{1-y}\Bigg).\label{defh}
  \end{align}
  (ii) For given $(p_{i,1})_{i\geq 1}$, we have a unique classical solution to
  \eqref{equation_P}-\eqref{defh} on $]0,1[\times]0,1[$.
\end{proposition}

\bigskip

The function $h$ in \eqref{defh} only depends on a boundary condition
($\partial^2P/\partial x\partial y$ at the boundaries $x=0$ or $y=0$, \ie the
$p_{i,1}$'s for $i\in \N^*$), which is non-classical, while the operator $A$ is of first order and hence
associated with some transport equations.

\begin{proof}[Proof of Proposition \ref{prop:eqP}]Let us first establish (i). Using
  the Markov property at time $t=1$:
  \begin{equation*}
    p_{i,j}=\frac{r}{2(r+d)}(p_{i+1,j}+p_{i,j+1})+\frac{dj}{(r+d)(i+j)} p_{i,j-1}+\frac{di}{(r+d)(i+j)}p_{i-1,j},
  \end{equation*}
  then multiplying by $x^i y^j$, and summing over $i,j\in \N^*$ leads to:
  \begin{multline}
    (r+d)\sum_{i,j\geq 1}(i+j)p_{i,j}x^i y^j =  \frac{r}{2} \sum_{i,j\geq 1}(i+j)(p_{i+1,j}+p_{i,j+1})x^i y^j\\
    + d \sum_{i,j\geq 1}j p_{i,j-1}x^i y^j +d \sum_{i,j\geq 1} i p_{i-1,j} x^i
    y^j.\label{etape1}
  \end{multline}
  The l.h.s.\ of (\ref{etape1}) equals:
  \begin{equation}
    (r+d)\sum_{i,j\geq 1}(i+j)p_{i,j}x^i y^j=(r+d)\bigg(x \frac{\partial}{\partial x}+y\frac{\partial }{\partial y}\bigg)P(x,y).\label{etape6}
  \end{equation}For the r.h.s.\ of (\ref{etape1}):
  \begin{multline}
    \begin{aligned} & \frac{r}{2}\sum_{\substack{j\geq 1\\ i\geq 2}}p_{i,j}(i-1+j)x^{i-1}y^j
    +\frac{r}{2}\sum_{\substack{i\geq 1\\ j\geq 2}}p_{i,j}(i+j-1) x^i y^{j-1}\\
   & + d\sum_{\substack{i\geq 1 \\ j\geq 0}}(j+1)p_{i,j}x^i y^{j+1}+d\sum_{\substack{i\geq 0\\ j\geq 1}}(i+1)p_{i,j}x^{i+1} y^j
   \end{aligned}\\
    \begin{aligned}
      = & \frac{r}{2}\sum_{\substack{j\geq 1\\ i\geq 2}}p_{i,j}(i-1)x^{i-1}y^j +
      \frac{r}{2}\sum_{\substack{j\geq 1\\ i\geq 2}}p_{i,j} j x^{i-1}y^j \\
       & +  \frac{r}{2}\sum_{\substack{i\geq 1\\ j\geq 2}}p_{i,j}i x^i y^{j-1}+\frac{r}{2}\sum_{\substack{i\geq 1\\ j\geq 2}}p_{i,j}(j-1) x^i y^{j-1}\\
       & +  d\sum_{\substack{i\geq 1 \\ j\geq 0}}(j+1)p_{i,j}x^i
      y^{j+1}+d\sum_{\substack{i\geq 0\\ j\geq 1}}(i+1)p_{i,j}x^{i+1}
      y^j\label{etape2}.
    \end{aligned}\end{multline}
  For the first term in the r.h.s.\ of (\ref{etape2}):
  \begin{align}
    \frac{r}{2}\sum_{\substack{j\geq 1\\ i\geq 2}}p_{i,j}(i-1)x^{i-1}y^j= &
    \frac{r}{2}\sum_{\substack{j\geq 1\\ i\geq 1}}p_{i,j}(i-1)x^{i-1}y^j
    =  \frac{r}{2}\sum_{\substack{j\geq 1\\ i\geq 1}}p_{i,j}ix^{i-1}y^j-\frac{r}{2}\sum_{\substack{j\geq 1\\ i\geq 1}}p_{i,j}x^{i-1}y^j\nonumber\\
    = & \frac{r}{2}\frac{\partial}{\partial
      x}P(x,y)-\frac{r}{2}\frac{P(x,y)}{x}.\label{etape3}
  \end{align}Similar computation holds for the \nth{4} term of the r.h.s.
  of (\ref{etape2}). For the \nth{2} term:
  \begin{align}
    \frac{r}{2}\sum_{\substack{j\geq 1\\ i\geq 2}}p_{i,j} j x^{i-1}y^j= &
    \frac{r}{2}\frac{y}{x}\sum_{\substack{j\geq 1\\ i\geq 2}}p_{i,j}j x^i
    y^{j-1}
    =  \frac{r}{2}\frac{y}{x}\Bigg(\sum_{\substack{j\geq 1\\ i\geq 1}}p_{i,j} j x^{i}y^{j-1}-\sum_{j\geq 1}p_{1,j}j x y^{j-1}\Bigg)\nonumber\\
    = & \frac{r}{2}\frac{y}{x}\frac{\partial}{\partial
      y}P(x,y)-\frac{r}{2}\sum_{j\geq 1}p_{1,j}j y^j.\label{etape4}
  \end{align}We handle the \nth{3} term of the r.h.s. of
  (\ref{etape2}) similarly. Now for the \nth{5} term:
  \begin{align}
    d \sum_{\substack{i\geq 1 \\ j\geq 0}}p_{i,j} (j+1)
    x^{i}y^{j+1} 
    = & d \Bigg(y^2 \frac{\partial}{\partial y}P(x,y)+ y P(x,y)+\sum_{i\geq
      1}x^i y\Bigg).\label{etape5}
  \end{align}Similar computation holds for the last term of (\ref{etape2}). From
  (\ref{etape6}), (\ref{etape2}), (\ref{etape3}), (\ref{etape4}) and
  (\ref{etape5}) we deduce that:
  \begin{multline*}
    (r+d)\bigg(x \frac{\partial}{\partial x}+y\frac{\partial }{\partial
      y}\bigg)P(x,y)= \frac{r}{2}\frac{\partial}{\partial
      x}P(x,y)-\frac{r}{2}\frac{P(x,y)}{x}
    +  \frac{r}{2}\frac{\partial}{\partial y}P(x,y)-\frac{r}{2}\frac{P(x,y)}{y}\\
    \begin{aligned}
      + & \frac{r}{2}\frac{y}{x}\frac{\partial}{\partial
        y}P(x,y)-\frac{r}{2}\sum_{j\geq 1}p_{1,j}j y^j
      +  \frac{r}{2}\frac{x}{y}\frac{\partial}{\partial x}P(x,y)-\frac{r}{2}\sum_{i\geq 1}p_{i,1}i x^i\\
      + & d \Bigg(y^2 \frac{\partial}{\partial y}P(x,y)+ y P(x,y)+\sum_{i\geq
        1}x^i y\Bigg) + d \Bigg(x^2 \frac{\partial}{\partial x}P(x,y)+ x
      P(x,y)+\sum_{j\geq 1}x y^j\Bigg),
    \end{aligned}
  \end{multline*}
and finally $\sum_{i\geq 1}p_{i,1}ix^i=x\frac{\partial^2P}{\partial x\partial y}(x,0)$.\\
  For point (ii),  local existence and uniqueness  stem from classical theorems
  \cite{zachmanoglou}.
Note that we construct an explicit, albeit complicated solution
\eqref{explicit_one} using the method of characteristics.
This concludes the proof.
\end{proof}

\bigskip

\begin{Rque}
  In the case of homogeneous random walks, the operator $AP(x,y)$ has the
  product form $R(x,y)P(x,y)$, see \cite{kurkovaraschel}. In some sense, this
  means that the inhomogeneity leads to partial derivatives in the functional
  equation.
\end{Rque}
\begin{Rque}
  This technique allowing to compute the solution of a discrete, linear problem
  thanks to generating series is also called the Z-transform method.
\end{Rque}

\subsection{Characteristic curves}\label{section:caracteristiques}

In Sections \ref{section:caracteristiques} and \ref{section:use-characteristics}
we establish, by the methods of characteristic equations, an explicit formula
for the solutions to \eqref{equation_P}, which proves the existence of the solution. Since $A$ is a first-order differential operator, we have a transport-like PDE. We introduce the following characteristic ordinary differential equations (ODEs). Let $(x_s,y_s)_{s\in \R_+}$ be the solution to the system:
\begin{equation}\label{system_xy}
\left\{\begin{array}{ccc}
\displaystyle \dot{x}_s&=& \displaystyle \frac{\text{d}x}{\text{d}s}(s)\ =\ Q(x_s,y_s),\phantom{{\widetilde{I}}}\\
\displaystyle \dot{y}_s&=&\displaystyle \frac{\text{d}y}{\text{d}s}(s)\ =\ Q(y_s,x_s),\phantom{\widetilde{\widetilde{\widetilde{\widetilde{I}}}}}
\end{array}\right.
\end{equation}
where $Q$ has been defined in (\ref{def:A}). The dynamical system (\ref{system_xy}) and its solutions will turn out to be decisive in the sequel---\eg in Proposition \ref{prop_resol_avec_sol_caract}, where we will use these characteristic
equations in order to express the solutions to the fundamental functional equation (\ref{equation_P}).\\

\begin{proposition}\label{prop:ODExy}
For any initial condition $(x_0,y_0)\in\mathbb{R}^{2}$ such that $x_0,y_0\neq 0$,
there exists a unique solution to \eqref{system_xy}, defined for $s\in \R$ by:
\begin{align}
x_s=  \frac{\lambda r e^{rs}+\mu d\, e^{ds}}{d(\lambda e^{rs}+\mu e^{d\,s}+1)}, \qquad
 y_s=  \frac{\lambda r e^{rs}+\mu d\, e^{ds}}{d(\lambda e^{rs}+\mu e^{d\,s}-1)},\label{equation_xys}
\end{align}with:
\begin{align}
\lambda=\frac{2d\, x_0 y_0-d(x_0+y_0)}{(x_0-y_0)(r-d)} , \qquad \mu=\frac{-2d\, x_0 y_0+r(x_0+y_0)}{(x_0-y_0)(r-d)}.\label{eq:lambdamu}
\end{align}
\end{proposition}

\begin{figure}[!ht] \begin{center}
\includegraphics[width=5cm,trim=5cm 5cm 5cm 5cm]{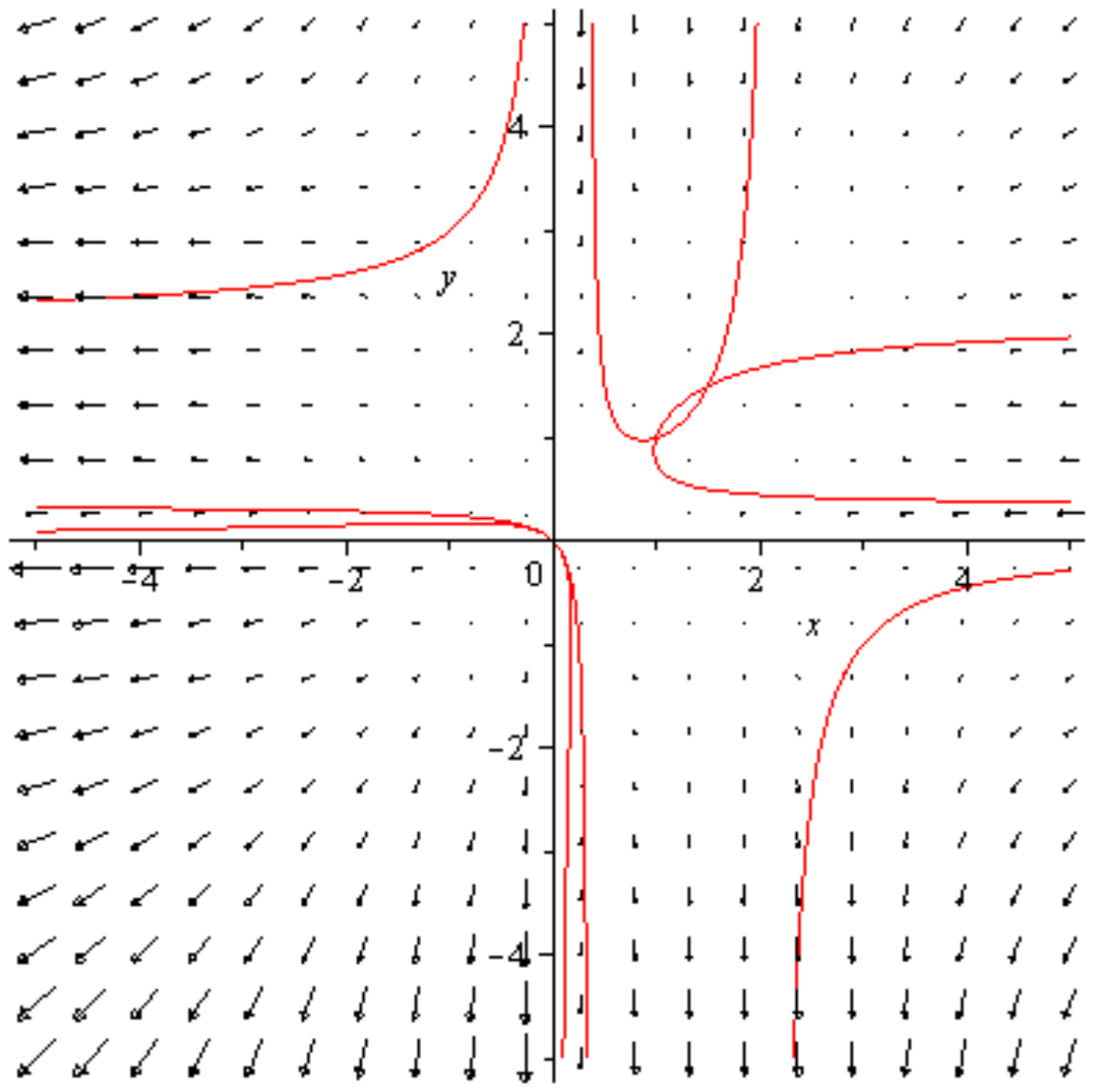}
\vspace{-4cm}\caption{{ \textit{Differential system (\ref{system_xy}): Vector field $(Q(x,y),Q(y,x))$. The plain lines correspond to the sets $\{Q(x,y)=0\}$ or $\{Q(y,x)=0\}$.}}}
\label{table_dynamic}
\end{center}\end{figure}

\vspace{-1cm}
\begin{figure}[!ht] \begin{center}
\begin{tabular}{cc}(a)\vspace{0.8cm} & (b)\vspace{0.8cm}\\
\includegraphics[width=5cm,trim=5cm 5cm 5cm 5cm]{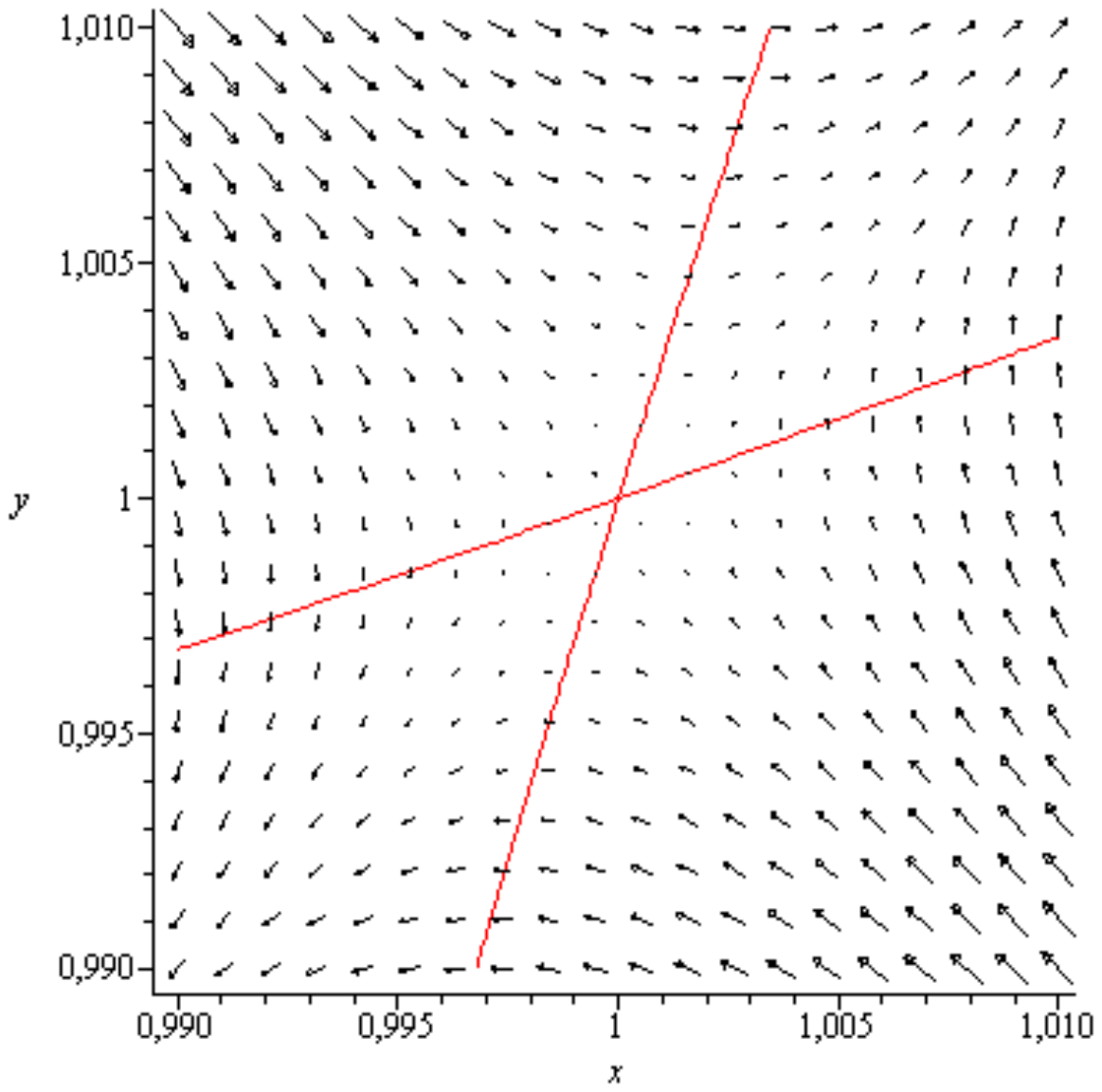}&\includegraphics[width=5cm,trim=5cm 5cm 5cm 5cm]{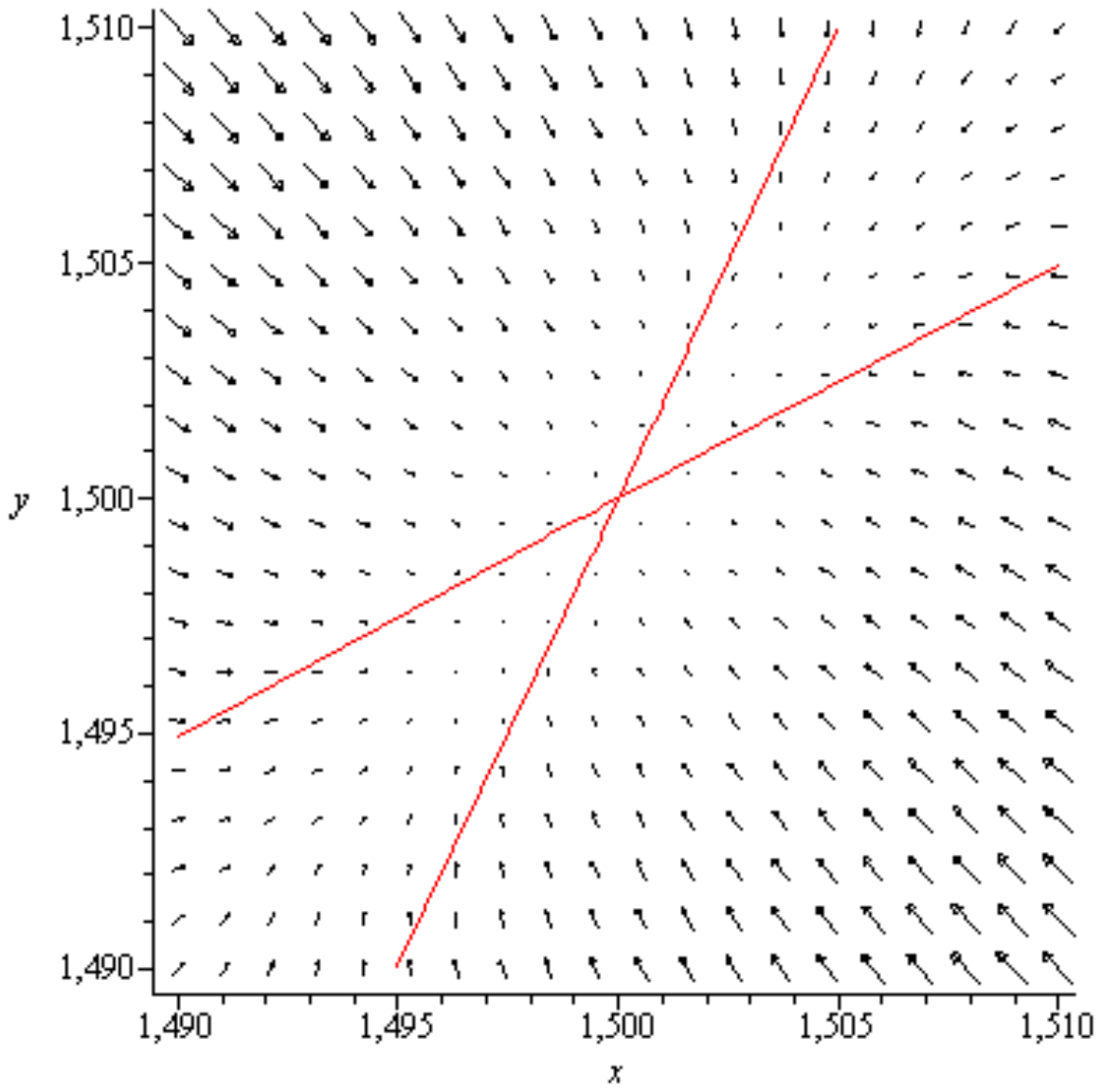}\end{tabular}
\vspace{-4cm}\caption{{ \textit{Differential system (\ref{system_xy}): (a) Neighborhood of the saddle point (1,1). (b) Neighborhood of the attractive equilibrium $(r/d,r/d)$.}}}
\label{table_dynamic2}
\end{center}\end{figure}

\begin{proof}First of all, since $Q(x,y)$ and $Q(y,x)$ are locally Lipschitz-continuous for $x,y\not= 0$, there is local existence and uniqueness. Let us introduce the new variable $z_s=x_s/y_s$, which is defined as long as $y_s\not=0$. Using the expression \eqref{def:A} of $Q$, the system (\ref{system_xy}) becomes:
\begin{align}
  \dot{x}_s= &(r+d)x_s-\frac{r}{2}(1+z_s)-dx^2_s,\label{system_xz(x)}\\
  \dot{z}_s= & \frac{\dot{x}_s y_s-x_s\dot{y}_s}{y^2_s} \nonumber\\
 = & (r+d) z_s - \frac{r}{2}\frac{z_s}{x_s}-\frac{r}{2}\frac{z_s}{y_s}-dx_s z_s-(r+d)z_s+\frac{r}{2}\frac{z_s }{y_s }+\frac{r}{2}\frac{z_s }{x_s }+dy_s z_s \nonumber\\
 = & d\,x_s (1-z_s ). \label{system_xz}
\end{align}
From (\ref{system_xz}), we obtain:
\begin{equation}
x_s =   \frac{\dot{z}_s }{d(1-z_s )},\label{equation_x}
\end{equation}
and differentiating (\ref{system_xz}) with respect to the time $s$ gives:
\begin{equation*}
\dot{x}_s = \frac{d\ddot{z}_s (1-z_s )+d\dot{z}^2_s }{d^2(1-z_s )^2}.
\end{equation*}Plugging this expression and (\ref{equation_x}) into (\ref{system_xz(x)}) provides:
\begin{equation*}
  \frac{d\ddot{z}_s (1-z_s )+d\dot{z}^2_s }{d^2(1-z_s )^2}=\frac{(r+d)\dot{z}_s }{d(1-z_s )}-\frac{r}{2}(1+z_s )-d \frac{\dot{z}^2_s }{d^2(1-z_s )^2},
 \end{equation*}
 and, therefore,
\begin{equation}
     \ddot{z}_s (1-z_s )+2 \dot{z}^2_s -(r+d)(1-z_s )\dot{z}_s +\frac{rd}{2}(1+z_s )(1-z_s )^2=0.\label{equation_z}
\end{equation}
The system (\ref{equation_x})--(\ref{equation_z}) can be solved explicitly. Let us define
$
u_s=({1+z_s})/({1-z_s})$, so that \begin{equation} z_s=\frac{u_s-1}{u_s+1},\label{rel_uz}
\end{equation}from which we obtain:
\begin{align}
\dot{z}_s=\frac{2\dot{u}_s}{(1+u_s)^2},\qquad \ddot{z}_s=\frac{2\ddot{u}_s(1+u_s)-4\dot{u}_s^2}{(1+u_s)^3}.
\end{align}Using these expressions in (\ref{equation_z}) provides:
\begin{align}
 & \frac{2\ddot{u}_s(1+u_s)-4\dot{u}_s^2}{(1+u_s)^3}\frac{2}{u_s+1}+2\frac{4\dot{u}^2_s}{(1+u_s)^4}\nonumber\\
 & \vspace{2cm}-(r+d)\frac{2}{u_s+1}\frac{2\dot{u}_s}{(1+u_s)^2}
 +\frac{rd}{2}\frac{2u_s}{u_s+1}\frac{4}{(u_s+1)^2}=0\nonumber\\
\Leftrightarrow \quad & 4 \ddot{u}_s (1+u_s)-4(r+d) \dot{u}_s(1+u_s)+4rd\, u_s(u_s+1)=0\nonumber\\
\Leftrightarrow \quad & \ddot{u}_s -(r+d)\dot{u}_s+rd\, u_s=0.
\end{align}Hence $u$ satisfies a second-order ordinary differential equation, that solves in:
\begin{equation}
u_s=\lambda e^{rs}+\mu e^{d\,s},\qquad \lambda,\mu\in \R.\label{soluce_u}
\end{equation}From (\ref{equation_x}), (\ref{rel_uz}) and (\ref{soluce_u}):
\begin{align}
z_s=  \frac{\lambda e^{rs}+\mu e^{d\,s}-1}{\lambda e^{rs}+\mu e^{d\,s}+1},\qquad
x_s=  \frac{\lambda r e^{rs}+\mu d\, e^{ds}}{d(\lambda e^{rs}+\mu e^{d\,s}+1)}.\label{soluce_zx}
\end{align}The integration constants $\lambda$ and $\mu$ can be expressed in terms of the initial conditions $x_0$ and $z_0=x_0/y_0$:
\begin{align}
\lambda=\frac{2d\, x_0-d(1+z_0)}{(1-z_0)(r-d)} , \qquad \mu=\frac{r-2dx_0+z_0r}{(1-z_0)(r-d)}.
\end{align}This yields the announced result with $y_s=x_s/z_s$.\end{proof}

\begin{Rque}
Explicit expressions for $x_s$ and $y_s$ as functions of time and parameterized by $r$ and $d$ can be obtained. Using the expressions \eqref{eq:lambdamu} of $\lambda$ and $\mu$ in \eqref{soluce_zx}
and the relations between $x_s$, $y_s$ and $z_s$ provides:
\begin{align}
x_s= & \frac{-\dfrac{x_0+y_0-2x_0y_0}{x_0+y_0-2(d/r)x_0y_0}\exp(rs)+ \exp(ds)}{-\dfrac{d}{r}\dfrac{x_0+y_0-2x_0y_0}{x_0+y_0-2(d/r)x_0y_0}\exp(rs)+\exp(ds)+\dfrac{(1-d/r)(y_0-x_0)}{x_0+y_0-2(d/r)x_0y_0}},
\label{explicit_x}\\
y_s= & \frac{-\dfrac{x_0+y_0-2x_0y_0}{x_0+y_0-2(d/r)x_0y_0}  \exp(rs)+ \exp(ds)}{-\dfrac{d}{r}\dfrac{x_0+y_0-2x_0y_0}{x_0+y_0-2(d/r)x_0y_0}\exp(rs)+\exp(ds)-\dfrac{(1-d/r)(y_0-x_0)}{x_0+y_0-2(d/r)x_0y_0}}
\label{explicit_y}.
\end{align}
For the sequel, it is useful to notice that the constant
     \begin{equation}
     \label{double_ine}
          -\frac{\lambda \ r}{\mu \ d}=\frac{x_0+y_0-2x_0y_0}{x_0+y_0-2(d/r)x_0y_0}
     \end{equation}which appears in (\ref{explicit_x}) and (\ref{explicit_y}) belongs to $]0,1[$ for all $(x_{0},y_{0})\in]0,1[^{2}$; this is due to the fact that $d/r<1$.
\end{Rque}

Below, we list some properties of the solutions to the dynamical system \eqref{system_xy}. The trajectories of the solutions to \eqref{system_xy} can be decomposed into four steps.
In order to describe them, let us introduce
     \begin{equation}
     \label{def_s_0}
          s_0=\frac{1}{r-d}\log\left(-\frac{\mu \ d}{\lambda \ r}\right)=\frac{\log\left(\dfrac{x_0+y_0-2x_0y_0}{x_0+y_0-2(d/r)x_0y_0}\right)}{d-r}
     \end{equation}
and $s_\pm$ as the only positive roots of the denominators of $x_s$ and $y_s$ in \eqref{equation_xys}:
     \begin{equation}
     \label{denom_x_y}
\lambda e^{rs}+\mu e^{d\,s}\pm 1.
     \end{equation}

\begin{proposition}Let $(x_0,y_0)\in]0,1[^{2}$.
     \begin{enumerate}
          \item \label{stationary_solutions} The stationary solutions to (\ref{system_xy}) are the
          saddle point $(1,1)$ and the attractive point $(r/d,r/d)$.
          \item \label{point2} $s_0\in]0,\infty[$ and $s_{\pm}\in ]s_0,\infty[$.
          \item $\inf\{s_{+},s_{-}\}=s_{+}$ (resp.\ $s_{-}$) if and only if $y_0<x_0$ (resp.\ $y_0>x_0$).
          \item \label{point4} $\lim_{s \uparrow s_{+}}x_{s}=-\infty$ and $\lim_{s \downarrow s_{-}}y_{s}=-\infty$.
     \end{enumerate}
The next points specify the behavior of the solution to \eqref{system_xy}.
     \begin{enumerate}
     \setcounter{enumi}{4}
          \item \label{to_s0} On $[0,s_0[$, $(x_s,y_s)$ belongs to $]0,1[^{2}$ and goes to $(x_{s_{0}},y_{s_{0}})=(0,0)$
          as $s\to s_0$.
          \item \label{to_inf} On $]s_0,\inf\{s_{+},s_{-}\}[$, $(x_s,y_s)$ goes decreasingly to
          $(x_{\inf\{s_{+},s_{-}\}},y_{\inf\{s_{+},s_{-}\}})$---by ``decreasingly'' we mean that
          both coordinates decrease.
          \item \label{to_sup} On $]\inf\{s_{+},s_{-}\},\sup\{s_{+},s_{-}\}[$, $(x_s,y_s)$ goes to
          $(x_{\sup\{s_{+},s_{-}\}},y_{\sup\{s_{+},s_{-}\}})$ decreasingly.
          \item \label{to_infty} On $]\sup\{s_{+},s_{-}\},\infty[$, $(x_s,y_s)$ goes decreasingly to
          $(r/d,r/d)$.
\end{enumerate}
\end{proposition}

\bigskip

\begin{proof}[Proof of \textit{Item} \ref{stationary_solutions}.]
To find the stationary solutions to (\ref{system_xy}), let us solve $\dot{x}=0$ and $\dot{y}=0$.
With (\ref{system_xy}), we get $Q(x,y)=0$ and $Q(y,x)=0$. This directly implies that $(x,y)=(1,1)$
or $(x,y)=(r/d,r/d)$, see Table \ref{table_dynamic}.
Let us now study the stability of these two equilibria.

At $(1,1)$, the Jacobian of (\ref{system_xy}) is:
     \begin{equation*}
          \mbox{Jac}(1,1)=\frac{r}{2}J-d I,
     \end{equation*}
where $J$ is the $2\times 2$ matrix full of ones and where $I$ is the $2\times 2$ identity matrix. The eigenvalues of $\mbox{Jac}(1,1)$ are $-d<0$ and $r-d>0$, associated with the eigenvectors $(1,-1)$ and $(1,1)$, respectively. By classical linearization methods (\eg \cite[Chap.\ 3]{verhulstbook}), we deduce that the point $(1,1)$ is a saddle point.

At $(r/d,r/d)$, the Jacobian of (\ref{system_xy}) is:
     \begin{equation*}
          \mbox{Jac}(r/d,r/d)=\frac{d}{2}J-r I.
     \end{equation*}
The eigenvalues are $-r<0$ and $d-r<0$, associated with the eigenvectors $(1,-1)$ and $(1,1)$, respectively.
The point $(r/d,r/d)$ is therefore attractive.
\end{proof}

\begin{figure}
\begin{center}
\includegraphics[width=8cm]{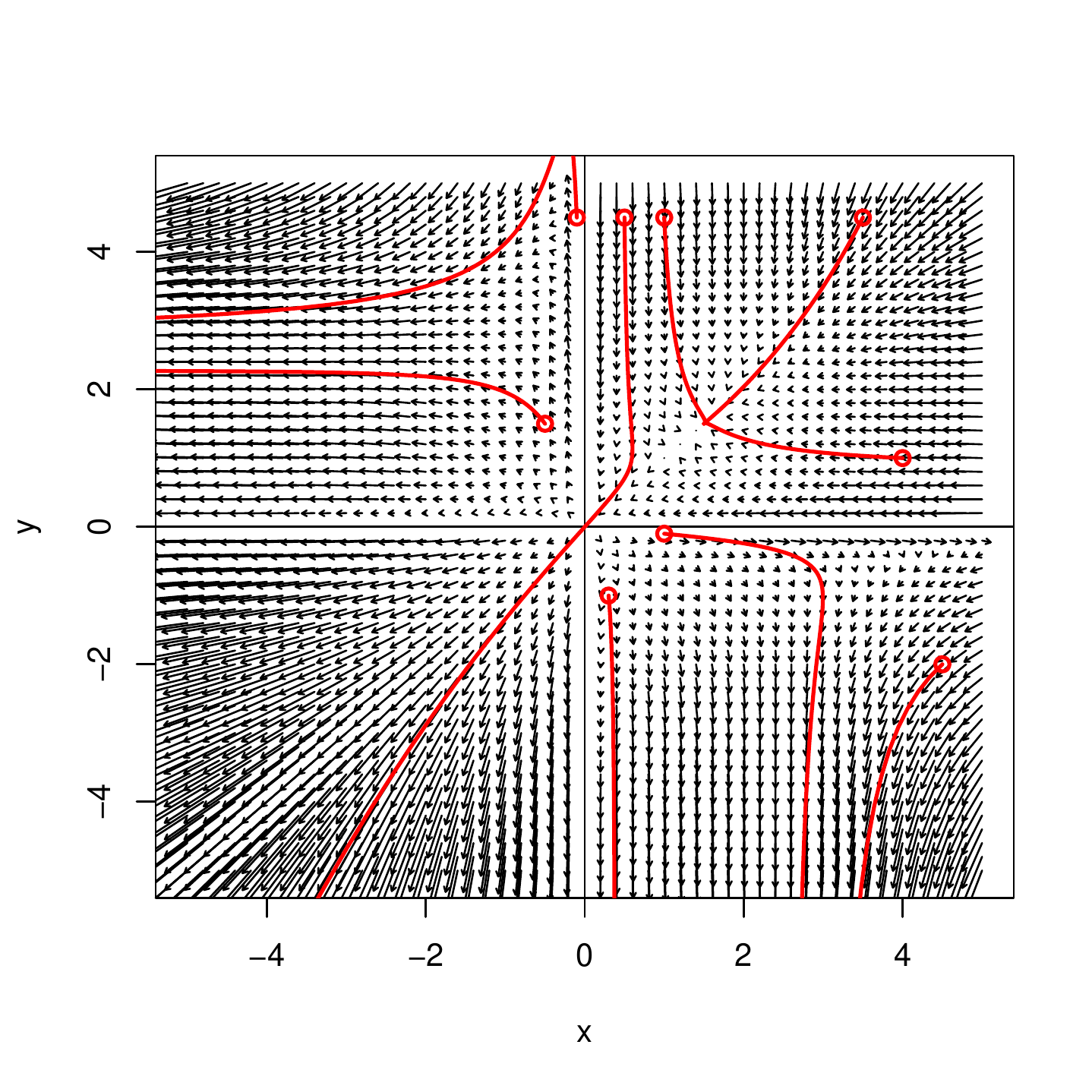}\vspace{-0.5cm}\caption{{\small
    \textit{Solutions corresponding to several initial conditions. We see that depending on the initial condition, the solutions converge of diverge to infinity.}}}
\end{center}
\end{figure}

\begin{proof}[Proof of \textit{Item} \ref{point2} to \ref{point4}]
Now we prove the different facts dealing with $s_{0}$, $s_{+}$ and $s_{-}$. First, \eqref{double_ine} and the fact that $r>d$
immediately imply that $s_0\in]0,\infty[$.
Next, we  show that (\ref{denom_x_y}) has on $[0,\infty[$ only one root, which belongs to $]s_0,\infty[$.
For this, we shall start with proving that (\ref{denom_x_y}) is positive on $[0,s_0]$. Then, we shall show that (\ref{denom_x_y}) is decreasing in $]s_0,\infty[$ and goes to $-\infty$ as $s\to \infty$.

In order to prove the first point above, it is enough to show that
(\ref{denom_x_y}) is positive at $s=0$ and increasing on $[0,s_0[$.
(\ref{denom_x_y}) is positive at $s=0$ simply because
     \begin{equation*}
\lambda +\mu \pm 1=\frac{x_0+y_0}{y_0-x_0}\pm 1=\frac{2y_0}{y_0-x_0} \mbox{ or }\frac{2x_0}{y_0-x_0}
          =\frac{(1-d/r)[x_0+y_0\pm (y_0-x_0)]}{x_0+y_0-2(d/r)x_0y_0}>0.
     \end{equation*}
To check that (\ref{denom_x_y}) is increasing on $[0,s_0[$, we note that the derivative of (\ref{denom_x_y})
is positive on $[0,s_0[$---actually by construction of $s_0$.

Now we prove the second point. From \eqref{double_ine} and since $r>d$, we obtain that
(\ref{denom_x_y}) goes to $-\infty$ as $s\to \infty$. Also, by definition of $s_0$, the derivative of (\ref{denom_x_y})
is negative on $]s_0,\infty[$, (\ref{denom_x_y}) is therefore decreasing on $]s_0,\infty[$.

The fact that $\inf\{s_{+},s_{-}\}$ equals $s_{+}$ (resp.\ $s_{-}$)
if and only if $y_0<x_0$ (resp.\ $y_0>x_0$) follows directly from
(\ref{denom_x_y}).

Finally, since the numerators of $x_s$ and $y_s$ are negative on
$]s_0,\infty[$, hence in particular at $s_\pm$, it is immediate
that $\lim_{s \uparrow s_{+}}x_{s}=-\infty$ and $\lim_{s \downarrow s_{-}}y_{s}=-\infty$.
\end{proof}

\begin{proof}[Proof of \textit{Item} \ref{to_s0} to \ref{to_infty}]
Let us first consider \textit{Item} \ref{to_s0}. By definition of $s_0$, the numerators
of $x_s$ and $y_s$ in (\ref{explicit_x}) and (\ref{explicit_y})
vanish for the first time at $s_0$. Moreover, since
$s_\pm>s_0$, both denominators are non-zero at $s_0$ and $x_{s_0}=y_{s_0}=0$.
In particular, on $[0,s_0[$, we have $(x_s,y_s)\in \mathbb{R}_{+}^{2}$. In fact,
$(x_s,y_s)\in ]0,1[^{2}$. Indeed, on the segment $\{1\}\times ]0,1[$,
$\dot{x}_s<0$ whereas on $]0,1[\times \{1\}$, $\dot{y}_s<0$:
it is therefore not possible to go through these segments.

We turn to the proof of \textit{Item} \ref{to_inf}. Thanks to (\ref{explicit_x}) and (\ref{explicit_y}), just after the time $s_0$,
$(x_s,y_s)$ belongs to the negative quadrant $\mathbb{R}_{-}^{2}$. But for any $(x,y)\in\mathbb{R}_{-}^{2}$,
$Q(x,y)\leq 0$ and $Q(y,x)\leq0$, see Table \ref{table_dynamic}, in such a way
that both $x_s$ and $y_s$ are decreasing as soon as they stay in this quarter plane,
in other words for $s\in ]s_0,\inf\{s_{+},s_{-}\}[$. At time $\inf\{s_{+},s_{-}\}$,
one (or even the two if $s_+=s_-$, \ie if $x_0=y_0$) of $x_s$ and $y_s$ becomes infinite.
In the sequel, let us assume that $\inf\{s_{+},s_{-}\}=s_+$; a
similar reasoning would hold for the symmetrical case $\inf\{s_{+},s_{-}\}=s_-$.

Let us show \textit{Item} \ref{to_sup}. Just after $s_+$, $(x_s,y_s)\in (\mathbb{R}_{+}\times \mathbb{R}_{-})\cap \{(x,y)\in \mathbb{R}^{2}
:\, Q(x,y)<0,\, Q(y,x)<0\}$. The latter set is simply connected and bounded by the curve $\{(x,y)\in
\mathbb{R}^{2} :\, Q(x,y)=0\}$, see Table \ref{table_dynamic}. Using classical arguments
(see \eg \cite{verhulstbook}), we obtain that it is not possible to go through this limiting curve
on which $\dot{x}_{s}=0$;
this is why for any $s\in]s_{+},s_{-}[$, $(x_s,y_s)$ remains inside of this set.

We conclude with the proof of \textit{Item} \ref{to_infty}. Just after the time $s_{-}$, $(x_s,y_s)\in \mathbb{R}_{+}^{2}\cap \{(x,y)\in \mathbb{R}^{2} :\, Q(x,y)<0,\, Q(y,x)<0\}$.
For the same reasons as above, $(x_s,y_s)$ cannot leave this set and actually converges to $(r/d,r/d)$.
\end{proof}

\subsection{Use of the characteristic curves to simplify the functional equation}\label{section:use-characteristics}

Let us assume the existence of a solution $P(x,y)$ to (\ref{equation_P}), and let us define
$g_s =P(x_s ,y_s ).$
Then:
\begin{multline*}
\dot{g}_s =\frac{dg}{ds}(s) =  \frac{\partial P}{\partial x}(x_s ,y_s ) \frac{\text{d}x_s}{\text{d}s}+\frac{\partial P}{\partial y}(x_s ,y_s ) \frac{\text{d}y_s}{\text{d}s}\\
= \frac{\partial P}{\partial x}(x_s ,y_s )Q(x_s ,y_s )+\frac{\partial P}{\partial y}(x_s ,y_s )Q(y_s ,x_s ).
\end{multline*}
Thus, if $P$ is a solution to (\ref{equation_P}), then:
\begin{equation}
     \label{first-formulation-ODE}
     \dot{g}_s +R(x_s ,y_s )g_s =h(x_s ,y_s ,P),
\end{equation}
which looks like a first-order ODE for $g$, except that $h$ depends on the boundary condition of $P$.

We first freeze the dependence on the solution in $h$, \ie we solve the ODE \eqref{first-formulation-ODE} as if the term in the right-hand side were a known function. Using the solutions to the characteristic equations, we shall prove the following result:\\

\begin{proposition}\label{prop_resol_avec_sol_caract}
Let $h(x,y)$ be an analytical function on $[0,1[^2$. Let $(x_0,y_0)\in \R^2$. The solution to the ODE
\begin{equation}\label{ODE1}
\dot{g}_s +R(x_s ,y_s )g_s =h(x_s ,y_s ),\qquad g_0=P(x_0,y_0),\end{equation}where $(x_s,y_s)_{s\geq 0}$ are the solutions \eqref{equation_xys} to the characteristic curve starting at $(x_0,y_0)$, is given by
\begin{align}
g^h_s =&   P(x_0,y_0)\exp\left(-\int_0^s R(x_u ,y_u )\textnormal{d}u\right)+\int_0^s h(x_u ,y_u )e^{-\int_u^s R(x_\alpha ,y_\alpha )\textnormal{d}\alpha}\textnormal{d}u\nonumber\\
 =:  &  F(s,x_0,y_0,h).\label{solution_avec_h}
\end{align}
\end{proposition}

\bigskip

\begin{proof}Equation \eqref{ODE1} is an inhomogeneous first-order ODE. The solution to the associated homogeneous equation is:
\begin{equation*}
g_s =R(x_0 ,y_0 )\exp\left(-\int_0^s R(x_u ,y_u )\textnormal{d}u\right).
\end{equation*}The announced result is deduced from the variation of constant method.
\end{proof}

A solution $P$ to (\ref{equation_P}) hence satisfies the following functional equation for all $s$, $x_0$ and $y_0$:
\begin{align}
P(x_s ,y_s )= & P(x_0,y_0)e^{-\int_0^s R(x_u ,y_u )\textnormal{d}u}
+  \int_0^s h(x_u ,y_u ,P)e^{-\int_u^s R(x_\alpha ,y_\alpha )\textnormal{d}\alpha}\textnormal{d}u,\label{equation_fonctionnelle}
\end{align}with the function $h$ defined in (\ref{defh}). Plugging the definitions (\ref{defP}) and (\ref{defh}) in (\ref{equation_fonctionnelle}), we obtain:
\begin{align}\lefteqn{\sum_{i,j\geq 1}p_{i,j}x^i_s  y^j_s =  P(x_0,y_0)e^{-\int_0^s R(x_\alpha ,y_\alpha )\textnormal{d}\alpha}}\nonumber\\
- &\frac{r}{2}\sum_{i\geq 1}p_{i,1}\int_0^s i(x_u )^i e^{-\int_u^s R(x_\alpha ,y_\alpha )\textnormal{d}\alpha}\textnormal{d}u
-  \frac{r}{2}\sum_{j\geq 1}p_{1,j}\int_0^s  j (y_u )^j e^{-\int_u^s R(x_\alpha ,y_\alpha )\textnormal{d}\alpha}\textnormal{d}u\nonumber\\
+ &
d\int_0^s x_u y_u  \left(\frac{1}{1-x_u }+\frac{1}{1-y_u }\right)e^{-\int_u^s R(x_\alpha ,y_\alpha )\textnormal{d}\alpha}\textnormal{d}u.\label{dev_eq_fonctionnelle}
\end{align}Notice that the r.h.s.\ of (\ref{dev_eq_fonctionnelle}) depends only on the $p_{i,1}$'s and $p_{1,j}$'s, while the l.h.s.\ depends on all $p_{i,j}$'s.\\

\begin{proposition}
\label{explicit_one}
Let $s_0>0$ be defined in (\ref{def_s_0}). We have:
\begin{eqnarray}
     P(x_0,y_0) &=& \frac{r}{2}\sum_{i\geq 1}p_{i,1}\int_0^{s_0} i(x_u )^i
     e^{\int_0^u R(x_\alpha ,y_\alpha )\textnormal{d}\alpha}\textnormal{d}u\nonumber\\
     & + & \frac{r}{2}\sum_{j\geq 1}p_{1,j}\int_0^{s_0}  j (y_u )^j e^{\int_0^u
     R(x_\alpha ,y_\alpha )\textnormal{d}\alpha}\textnormal{d}u\nonumber\\
     &-& d\int_0^{s_0} x_u y_u  \left(\frac{1}{1-x_u }+\frac{1}{1-y_u }\right)
     e^{\int_0^u R(x_\alpha ,y_\alpha )\textnormal{d}\alpha}\textnormal{d}u.\label{grosse-expression}
\end{eqnarray}
\end{proposition}

\bigskip

Before proving Proposition \ref{explicit_one}, let us show that the different
quantities that appear in its statement are well defined---indeed, this is
\textit{a priori} not clear: as $\alpha \to s_0$, $x_\alpha \to 0$ and $y_\alpha \to 0$,
in such a way that $R(x_\alpha ,y_\alpha )\to \infty$, see (\ref{def:A}).\\

\begin{lemma}
\label{i+j_large}
Let $i,j\in \mathbb{N}$. Then
     $
          \lim_{u\to s_0}(x_u )^i(y_u)^j
          e^{\int_0^u R(x_\alpha ,y_\alpha )\textnormal{d}\alpha}
     $
is finite if and only if $i+j\geq 1$---and equals zero if and only if $i+j\geq 2$.
\end{lemma}

\bigskip

\begin{proof}
First, since the only zero of any function of the form  $\alpha \exp(a s)+\beta \exp(b s)$ with $\alpha\beta<0$ and $a\neq b$ has order one, the following function has a simple zero at $s_0$:
     \begin{equation}\label{simple_zero}
          -\dfrac{x_0+y_0-2x_0y_0}
           {x_0+y_0-2(d/r)x_0y_0}\exp(ru)+ \exp(du).
     \end{equation}
Thanks to this and since $s_+,s_->s_0$, both $x_s$ and $y_s$ have a zero of order 1 at
$s_0$.

Moreover, with $\lambda$ and $\mu$ defined in \eqref{eq:lambdamu}, we obtain that:
     \begin{multline}
     \label{simplification_EIR}
     \exp\left(\int_0^u R(x_\alpha ,y_\alpha)\textnormal{d}\alpha\right) =
     \dfrac{\lambda/\mu+1-1/\mu}{(\lambda/\mu)\exp(r u)+\exp(d u)-1/\mu}\times \\
           \times \dfrac{\lambda/\mu+1+1/\mu}{(\lambda/\mu)\exp(r u)+\exp(d u)+1/\mu}
           \dfrac{r/d\lambda/\mu+1}{(r/d)(\lambda/\mu)\exp(ru)+ \exp(du)}\exp((r+d)u).
     \end{multline}
     It is indeed easy to check that the derivative of the logarithm of (\ref{simplification_EIR}) is
     equal to $R(x_u,y_u)$, for which we have an explicit expression, see (\ref{def:A}), (\ref{explicit_x})
     and (\ref{explicit_y}).

From \eqref{simplification_EIR}, we see that $e^{\int_0^u R(x_\alpha ,y_\alpha )\textnormal{d}\alpha}$ has three poles, namely at $s_0,s_+,s_-$.
The zero at $s_0$ has order one by using again the considerations on the zeros of \eqref{simple_zero}. In
particular, Lemma \ref{i+j_large} follows immediately.
\end{proof}

\bigskip
\begin{proof}[Proof of Proposition \ref{explicit_one}]
Start by multiplying (\ref{dev_eq_fonctionnelle}) by
$e^{\int_0^s R(x_\alpha ,y_\alpha )\textnormal{d}\alpha}$ and then
let $s\to s_0$. Since
     $
          P(x,y)=xy\sum_{i,j\geq 1}p_{i,j}x^{i-1}y^{j-1},
     $
see (\ref{defP}), and since
$\lim_{s\to s_0}x_s y_se^{\int_0^s R(x_\alpha ,y_\alpha )\textnormal{d}\alpha}=0$, see Lemma \ref{i+j_large}, we obtain that
     \begin{equation*}
          \lim_{s\to s_0}P(x_s,y_s)e^{\int_0^s R(x_\alpha ,y_\alpha )\textnormal{d}\alpha} = 0,
     \end{equation*}
which concludes the proof of Proposition \ref{explicit_one}.
\end{proof}

\bigskip
\begin{Rque}When $(x_0,y_0)\in (0,1)^2$, we also have $(x_u,y_u)\in (0,1)^2$ for all $u\in (0,s_0)$. Thus it is possible to plug approximations of the $p_{1,j}$'s and $p_{i,1}$'s into \eqref{grosse-expression} thanks to the terms $(x_u)^i$ and $(y_u)^j$. Using that $p_{1,i}= 2d/(ri)+o(1/i)$ when $i\rightarrow +\infty$, it is possible to find $I_0$ sufficiently large so that:
\begin{align*}
\lefteqn{ \frac{r}{2}\sum_{i> I_0}p_{i,1}\int_0^{s_0} i\big((x_u )^i+(y_u)^i\big)
     e^{\int_0^u R(x_\alpha ,y_\alpha )\textnormal{d}\alpha}\textnormal{d}u}\\
     \sim  &  d \int_0^{s_0} \sum_{i> I_0} \big((x_u )^i+(y_u)^i\big)
     e^{\int_0^u R(x_\alpha ,y_\alpha )\textnormal{d}\alpha}\textnormal{d}u\\
     = &
     d \int_0^{s_0} \Big(\frac{x_u^{I_0+1}}{1-x_u}+\frac{y_u^{I_0+1}}{1-y_u}\Big)e^{\int_0^u R(x_\alpha ,y_\alpha )\textnormal{d}\alpha}\textnormal{d}u\end{align*}
Thus:
\begin{multline}
     P(x_0,y_0) = \frac{r}{2}\sum_{i= 1}^{I_0} p_{i,1}\int_0^{s_0} i\big((x_u )^i+(y_u)^i\big)
     e^{\int_0^u R(x_\alpha ,y_\alpha )\textnormal{d}\alpha}\textnormal{d}u\\
      +  d \int_0^{s_0} \Big(\frac{x_u(x_u^{I_0}-y_u)}{1-x_u}+\frac{y_u(y_u^{I_0}-x_u)}{1-y_u}\Big)e^{\int_0^u R(x_\alpha ,y_\alpha)\textnormal{d}\alpha}\textnormal{d}u+o(x_0^{I_0+1}+y_0^{I_0+1}).
\end{multline}The latter expression shows that numerically, one can restrict to the computation of a finite number of probabilities $p_{i,1}$, for $i\leq I_0$.
\end{Rque}

\section{Numerical results}\label{section:numerique}

In this section, we present two different ways of approximating the extinction probabilities $p_{i,j}$.

\subsection{Probabilistic algorithm}\label{section:montecarlo}

A first possibility, if we are interested in a given initial condition $(i,j)$, is to approximate $p_{i,j}$ by Monte-Carlo simulations. For $T>0$ large, we simulate $M$ paths $(X^\ell_t,Y^\ell_t)_{t\in \{1,\dots,T\}}$ started at $(i,j)$, for $\ell \in \{1,\dots,M\}$, independent and distributed as the process $(X_t,Y_t)_{t\in \{1,\dots,T\}}$. The extinction probability is estimated by:
\begin{equation*}
\widehat{p}_{M,T}=\frac{1}{M}\sum_{\ell=1}^M \ind_{\{\exists t\leq T, X_t^\ell Y_t^\ell =0\}}.
\end{equation*}The estimator $\widehat{p}_{M,T}$ is the proportion of paths that have gone extinct before time $T$.\\

\begin{proposition}Let $(i,j)$ be the initial condition. The estimator $\widehat{p}_{M,T}$ has the following properties:
\begin{enumerate}
\item It is a convergent and unbiased estimator of $\P_{i,j}[\tau_0\leq T]$.
\item Its variance is
$\P_{i,j}[\tau_0\leq T](1-\P_{i,j}[\tau_0\leq T])/M$, and hence we have the following asymptotic 95\% confidence interval for $p_{i,j}$:
\begin{align}
\left[\widehat{p}_{M,T} - 1.96\sqrt{\frac{\widehat{p}_{M,T}(1-\widehat{p}_{M,T})}{M}} ; \widehat{p}_{M,T} + 1.96\sqrt{\frac{\widehat{p}_{M,T}(1-\widehat{p}_{M,T})}{M}}\right].\label{ic}
\end{align}
\end{enumerate}
\end{proposition}

\bigskip
\begin{proof}These results are straightforward consequences of the law of large numbers and central limit theorem, given that $\ind_{\{\exists t\leq T, X_t^\ell Y_t^\ell =0\}}$ are independent Bernoulli random variables with parameter $\P_{i,j}[\tau_0\leq T]$.
\end{proof}

Computing the extinction probabilities by Monte-Carlo methods yields good results if we are interested in a given initial condition $(i,j)$. We then have a complexity of order $M\times T$. However, biologists may be interested in investigating the extinction probabilities when the initial condition $(i,j)$ varies, and the method become computationally expensive.

\subsection{Deterministic algorithm}\label{section:calculnumerique}

For numerical approximations, we restrict ourselves to the computation of $(p_{i,j})_{i,j\in \{1,\dots,N\}}$ for a positive (large) integer $N$. In this case, \eqref{eq3} can be approximated by the solution to a linear system.

Let us define $\mathbf{p}_N=(p_{1,1},\dots,p_{1,N},p_{2,1},\dots, p_{2,N},\dots ,p_{N,1},\dots, p_{N,N})^T$ and $T_N$ is a $N^2\times N^2$-matrix with five non-zero diagonals:
\begin{align*}
T_N= & \left(\begin{array}{ccccc}
A_{1} & D & 0 & \dots &  0 \\
B_{2,1} & A_2  & D & \ddots & \vdots  \\
0 & B_{3,2} &  \ddots & \ddots &  0 \\
\vdots & \ddots & \ddots &  \ddots & D\\
0 & \dots &  0 & B_{N,N-1} & A_{N}
\end{array} \right)
\end{align*}
where $D=\frac{r}{2(r+d)}Id_N$, $A_i$ ($i\in \{1,\dots,N\}$) and $B_{i,i-1}$ ($i\in \{2,\dots N\}$) are the $N\times N$-matrices
\begin{align*}
&A_i=  \left(\begin{array}{ccccc}
- 1 & \frac{r}{2(r+d)} & 0 & \dots & 0\\
\frac{d}{r+d}\frac{2}{i+2} & -1 & \ddots & \ddots & \vdots \\
0 & \frac{d}{r+d}\frac{3}{i+3} & \ddots & \ddots & 0\\
\vdots & \ddots & \ddots & \ddots & \frac{r}{2(r+d)} \\
0 & \dots & 0 & \frac{d}{r+d}\frac{N}{i+N}& -1
\end{array}\right),\\
& B_{i,i-1}=\left(\begin{array}{cccc}
\frac{d}{r+d}\frac{i}{i+1} & 0  & \dots &  0 \\
0 & \frac{d}{r+d}\frac{i}{i+2} & \ddots  & \vdots\\
\vdots & \ddots & \ddots & 0\\
0 &  \dots & 0 & \frac{d}{r+d}\frac{i}{i+N}
\end{array}\right).\end{align*}
Let us also define the vector $b_N=(b_{1N},\dots,b_{NN})^T\in \R^{N\times N}$ such that:
\begin{align*}
& b_{1 N} = - \left(\begin{array}{c}-\frac{d}{r+d}\\
\frac{d}{r+d}\frac{1}{1+2}\\
 \vdots \\
\frac{d}{r+d}\frac{1}{1+(N-1)}\\
\frac{d}{r+d}\frac{1}{1+N} +\frac{r}{2(r+d)}\widetilde{p}_{1,N+1}\end{array}\right),
\\
&   b_{iN} = -\left(\begin{array}{c}\frac{d}{r+d}\frac{1}{1+i}\\
0\\
 \vdots \\
0 \\
\frac{r}{2(r+d)}\widetilde{p}_{i,N+1} \end{array}\right),
\qquad \mbox{for }i\in \{2,\dots,N-1\},\\
& b_{NN}=- \left(\begin{array}{c}
\frac{d}{r+d}\frac{1}{1+N}+\frac{r}{2(r+d)}\widetilde{p}_{N+1,1}\\
\frac{r}{2(r+d)}\widetilde{p}_{N+1,2}\\
 \vdots \\
\frac{r}{2(r+d)}\widetilde{p}_{N+1,N-1} \\
\frac{r}{2(r+d)}\big(\widetilde{p}_{N+1,N}+ \widetilde{p}_{N,N+1}\big)\end{array}\right)
\end{align*}where $\widetilde{p}_{i,N+1}$ and $\widetilde{p}_{N+1,N}$ are approximations of $p_{i,N+1}$ given by Proposition \ref{AsymptoticBehaviorAP}. With these notations, \eqref{eq3} rewrites as
\begin{equation*}
 T_N \mathbf{p}_N=b_N.
\end{equation*}

\subsection{Results}

We start with $r=3$ and $d=2$. For the Monte-Carlo simulation, we use $M=200$
and $T=5000$. For the deterministic method, we use $N=50$, so that $(i,j)\in
\{1,\dots,50\}^2$. Estimators of the extinction probabilities
$\widehat{p}^{(1)}_{i,j}$ and $\widehat{p}^{(2)}_{i,j}$ obtained respectively
from the methods of Sections \ref{section:montecarlo} and
\ref{section:calculnumerique} are plotted in Figure \ref{fig:r=3d=3}. The
results given by both methods are very similar, as shown by the statistics of
Table \ref{table:comparaison}. In Table \ref{table:comparaison}, we compute the
square difference between the two predictions
$(\widehat{p}^{(1)}_{i,j}-\widehat{p}^{(2)}_{i,j})^2$, the absolute difference
$|\widehat{p}^{(1)}_{i,j}-\widehat{p}^{(2)}_{i,j}|$ and the relative difference
$|\widehat{p}^{(1)}_{i,j}-\widehat{p}^{(2)}_{i,j}|/\widehat{p}^{(2)}_{i,j}$. For
the latter, we consider only the couples $(i,j)$ where $\widehat{p}^{(1)}_{i,j}$
and $\widehat{p}^{(2)}_{i,j}$ do not  vanish (else, the fraction is either not
defined or either 1 whatever the value of $\widehat{p}^{(1)}_{i,j}$).

\begin{figure}[!ht] \begin{center}
\begin{tabular}{cc}(a)& (b)\\\includegraphics[width=6cm]{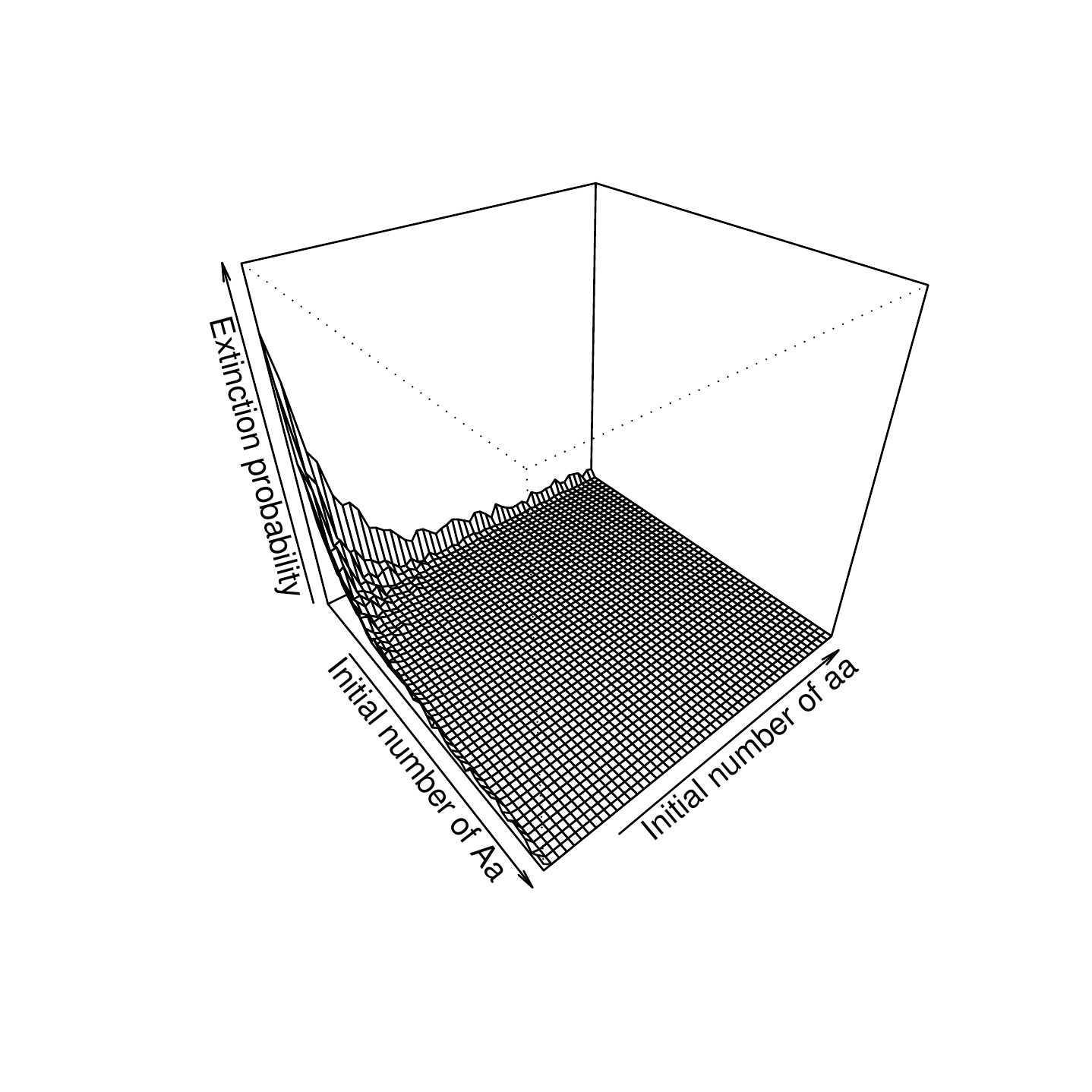}&\includegraphics[width=6cm]{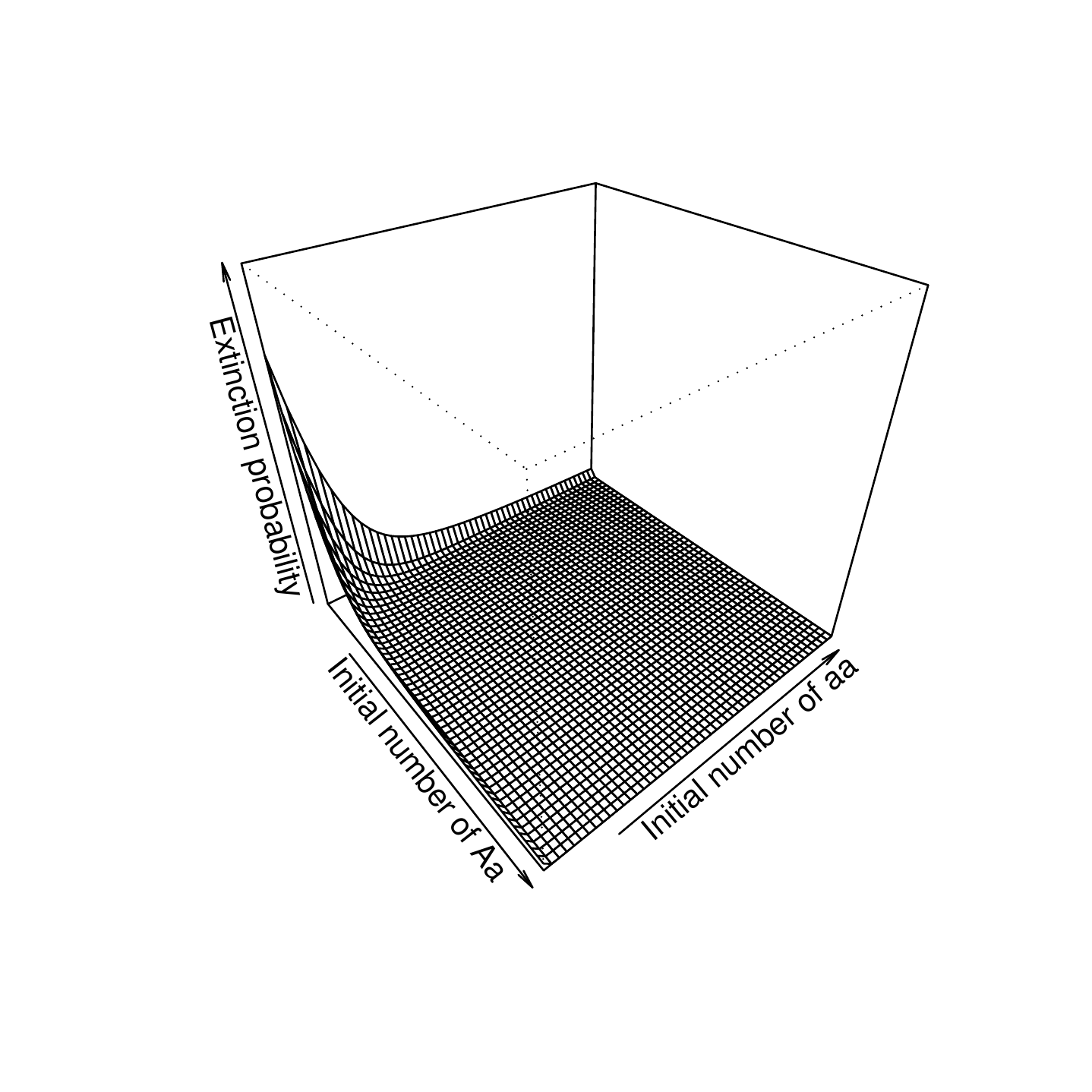}\end{tabular}\vspace{-0.5cm}
\caption{{ \textit{Estimation of the extinction probabilities $p_{i,j}$'s when $r=3$ and $d=2$: (a) with the Monte-Carlo method of Section \ref{section:montecarlo}. (b) with the deterministic method of Section \ref{section:calculnumerique}.}}}
\label{fig:r=3d=3}
\end{center}\end{figure}

\begin{table}[!ht]
\begin{center}
\begin{tabular}{|c|cccc|}
\hline
 & Mean & St.dev & Min & Max \\
 \hline
Square error & $3.24\ 10^{-5}$ & $2.47\ 10^{-4}$ & $1.68\ 10^{-36}$ & $4.63 \ 10^{-3} $ \\
Absolute error & $1.34\ 10^{-3}$ & $5.53\ 10^{-3}$ & $1.30\ 10^{-18}$ & $6.81 \ 10^{-2} $ \\
Relative error & $4.49\ 10^{-2}$ & $1.71\ 10^{-1}$ & $9.80\ 10^{-3}$ & $9.71\ 10^{-1}$  \\
\hline
\end{tabular}
\caption{{ \textit{Square, absolute and relative differences between the predictions of the stochastic method when $r=3$ and $d=2$ (Section \ref{section:montecarlo}) and of the deterministic methods (Section \ref{section:calculnumerique}). Recall that with $M=200$, the width of the confidence interval \eqref{ic} is $6.92\ 10^{-2}$.}}}
\label{table:comparaison}
\end{center}
\end{table}

\noindent To carry further the comparison of the stochastic and deterministic method, and to observe the influence of $N$ on the quality of the approximation, we compute the relative quadratic error
\begin{equation}\frac{\sqrt{\sum_{1\leq i,j\leq 10} \big(\widehat{p}^{(2)}_{i,j}-\widehat{p}^{(3)}_{i,j}\big)^2}}{\sqrt{\sum_{1\leq i,j\leq 10} \big(\widehat{p}^{(2)}_{i,j}\big)^2}} \mbox{ or }\frac{\sqrt{\sum_{1\leq i,j\leq 10} \big(\widehat{p}^{(2)}_{i,j}-\widehat{p}^{(3)}_{i,j}\big)^2}}{\sqrt{\sum_{1\leq i,j\leq 10} \big(\widehat{p}^{(3)}_{i,j}\big)^2}}\label{rqe}\end{equation}
when $\widehat{p}^{(2)}_{i,j}$ is the deterministic approximation for $N\in \{10,\dots,50\}$ and $\widehat{p}^{(3)}_{i,j}$ is either given by the deterministic approximation with $N=50$, or by the stochastic approximation $\widehat{p}^{(1)}_{i,j}$ with $M=200$. In the first case when $\widehat{p}^{(3)}_{i,j}=\widehat{p}^{(1)}_{i,j}$, the decrease in the quadratic errors stops around $N=18$ around $0.0891$. This corresponds roughly to the stochastic error of the law of large numbers \eqref{ic} which depends only on $M$. In the second case, when $\widehat{p}^{(3)}_{i,j}$ is the deterministic approximation with $N=50$, the relative quadratic errors decrease exponentially fast in $\exp(-0.6842\ N)$ ($R^2=99.92\%$).

\bigskip

In a second experiment, we choose $r=2.002$ and $d=2$. This case is more interesting in population ecology, since small populations are of interest when they are fragile and endangered species. For the Monte-Carlo simulation, we use $M=200$ and $T=5000$. For the deterministic method, we use $N=100$, so that $(i,j)\in \{1,\dots,100\}^2$. The estimated extinction probabilities $\widehat{p}^{(1)}_{i,j}$ and $\widehat{p}^{(2)}_{i,j}$ are plotted in Figure \ref{fig:r=2002d=3}, and statistics are computed in Table \ref{table:comparaison2}. Again, results from both methods are similar. This is confirmed by computing the relative quadratic errors, with the $\widehat{p}^{(2)}_{i,j}$'s obtained from the deterministic method and the $\widehat{p}^{(3)}_{i,j}=\widehat{p}^{(1)}_{i,j}$'s from the Monte-Carlo method. The decrease of this error is exponential with $N$ in $\exp(-0.0619\ N)$ ($R^2=98.98\%$) as shown in Figure \ref{fig:r=2002d=3}(c). It can be noticed that in this case, the per!
 formances of the Monte-Carlo method match better the one of the deterministic algorithm. This is due to the fact that Monte-Carlo methods fail to produce good estimates of small probabilities (see \cite{clemenconetal} and references therein).\\
When the probabilities $\widehat{p}^{(3)}_{i,j}$'s are given by the deterministic method with $N=50$ in \eqref{rqe}, we have as in the previous case ($r=3$) an exponential decrease of the relative quadratic error in $\exp(-0.1092\ N)$ ($R^2=95.07\%$).

\begin{figure}[!ht] \begin{center}
\begin{tabular}{ccc}(a)& (b) \\
\hspace{-1cm}\includegraphics[width=6cm]{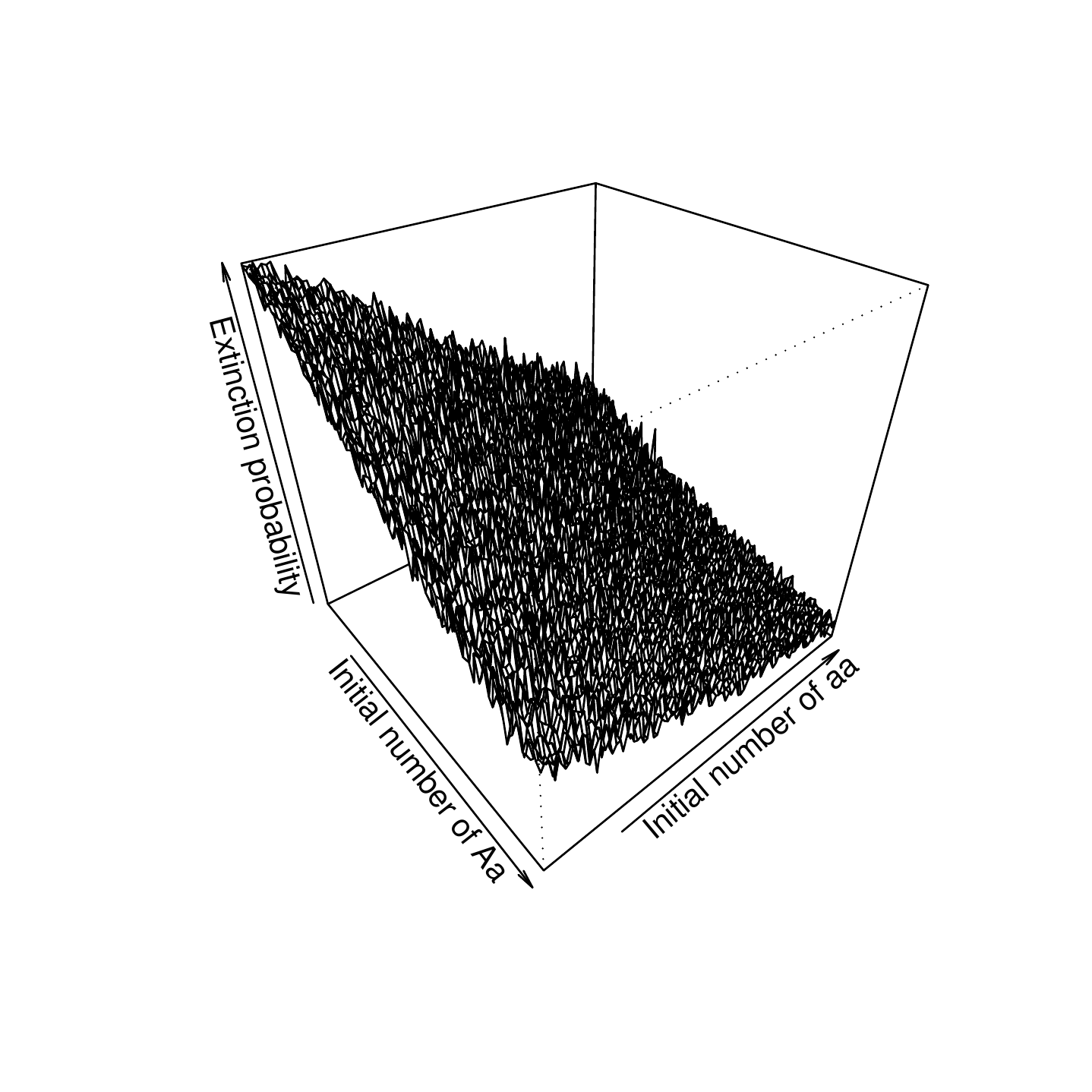}&
\hspace{-1cm}\includegraphics[width=6cm]{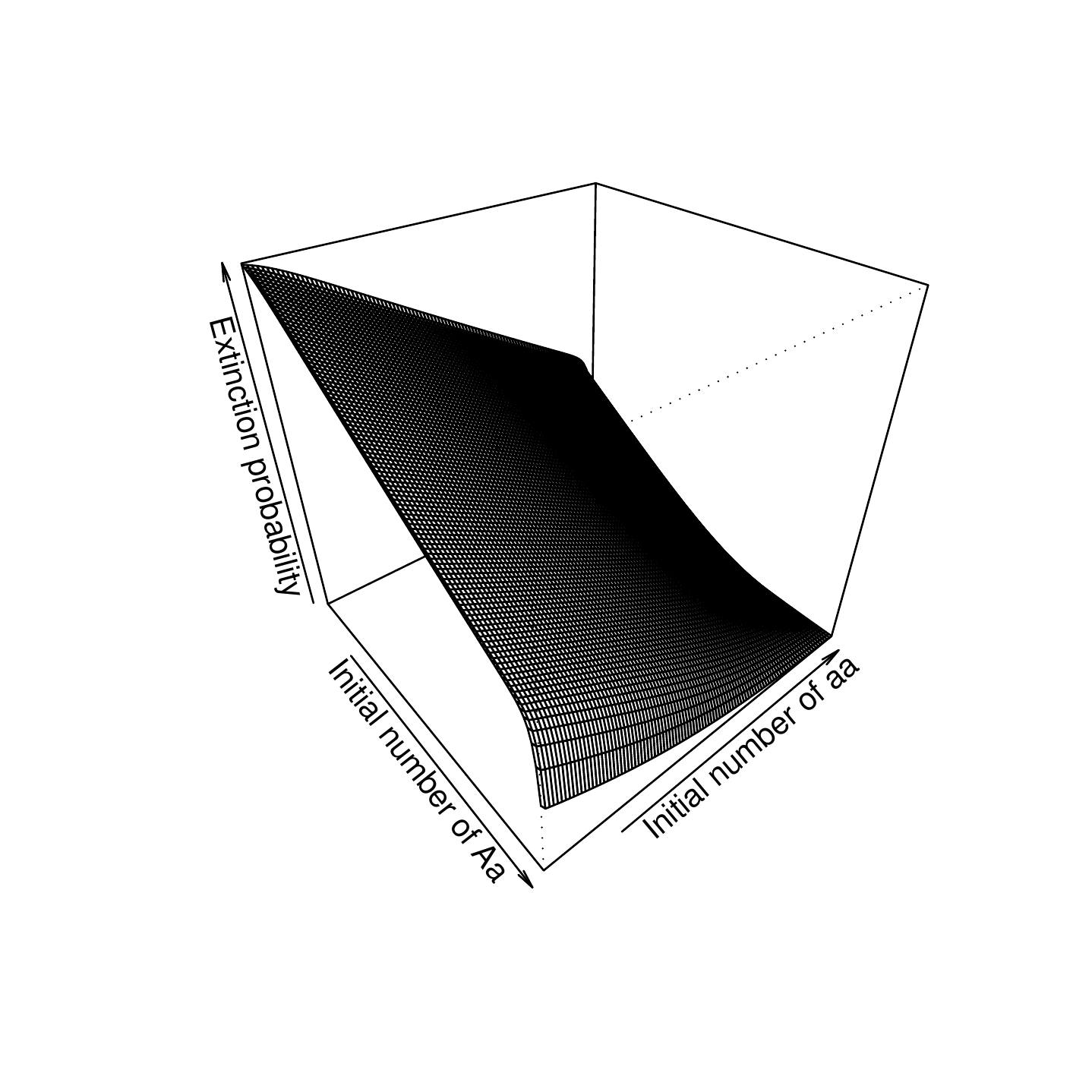} \end{tabular}
\vspace{-0.5cm}\caption{{ \textit{Estimation of the extinction probabilities $p_{i,j}$'s when $r=2.002$ and $d=2$: (a) with the Monte-Carlo method of Section \ref{section:montecarlo}. (b) with the deterministic method of Section \ref{section:calculnumerique}. }}}
\label{fig:r=2002d=3}
\end{center}\end{figure}

\begin{figure}[!ht] \begin{center}
\includegraphics[width=6cm]{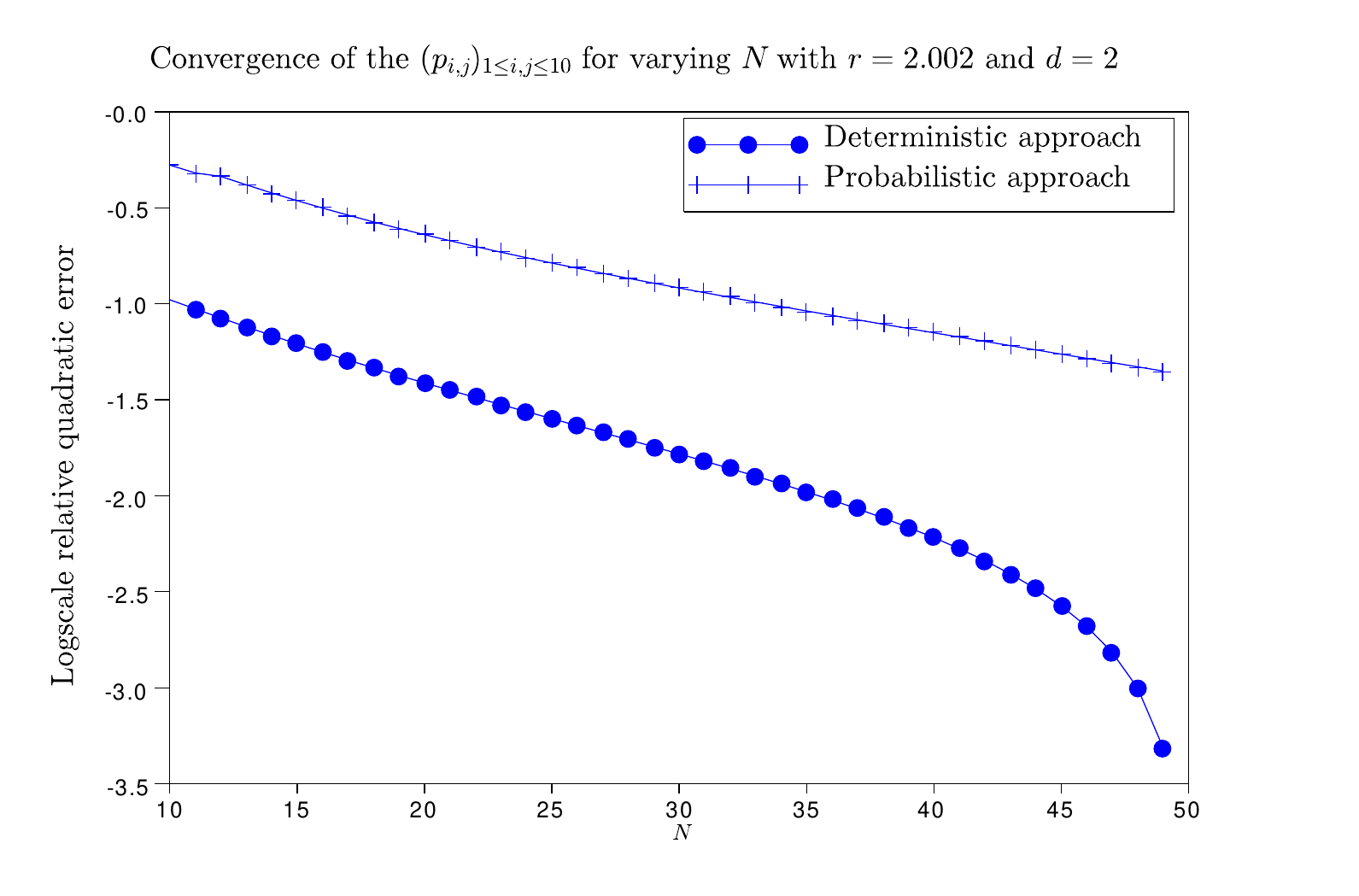}
\vspace{-0.5cm}\caption{{ \textit{Estimation of the extinction probabilities $p_{i,j}$'s when $r=2.002$ and $d=2$: Decrease with $N$ of the log of the relative quadratic errors \eqref{rqe} between the deterministic and Monte-Carlo methods ($M=200$).}}}
\label{fig:r=2002d=3bis}
\end{center}\end{figure}

\begin{table}[!ht]
\begin{center}
\begin{tabular}{|c|cccc|}
\hline
 & Mean & St.dev & Min & Max \\
 \hline
Square error & $6.27\ 10^{-3}$ & $6.25\ 10^{-3}$ & $8.42\ 10^{-10}$ & $5.31 \ 10^{-2} $ \\
Absolute error & $6.89\ 10^{-2}$ & $3.90\ 10^{-2}$ & $2.90\ 10^{-5}$ & $2.31 \ 10^{-1} $ \\
Relative error & $2.22\ 10^{-1}$ & $2.17\ 10^{-1}$ & $1.02\ 10^{-4}$ & $1$  \\
\hline
\end{tabular}
\caption{{ \textit{Square, absolute and relative differences between the predictions of the stochastic method when $r=2.002$ and $d=2$ (Section \ref{section:montecarlo}) and of the deterministic methods (Section \ref{section:calculnumerique}). As in Table \ref{table:comparaison}, since $M=200$, the width of the confidence interval \eqref{ic} is $6.92\ 10^{-2}$.}}}
\label{table:comparaison2}
\end{center}
\end{table}

\section*{Acknowledgements}
The authors warmly thank the referees for useful comments and suggestions.


\end{document}